\documentclass[reqn]{amsart}
\pdfoutput=1
\usepackage{dsfont}
\usepackage{bbm}
\usepackage[nodate]{datetime}
\usepackage[hmargin=31mm,top=25mm,bottom=25mm,a4paper]{geometry}
\usepackage{color}
\usepackage{graphicx}
\usepackage{subfig}
\usepackage{fancyhdr}
\usepackage{amsmath,amssymb,amsfonts,amsthm}
\usepackage{amstext}
\usepackage{amsmath}
\usepackage{amssymb}
\usepackage{enumerate}
\usepackage{amsbsy}
\usepackage{amsopn}
\usepackage{bbm,amsthm}
\usepackage{amscd}
\usepackage{amsxtra}
\usepackage{todonotes}

\newtheorem{theorem}{Theorem}[section]
\newtheorem*{theorem*}{Theorem}
\newtheorem{lemma}{Lemma}[section]
\newtheorem{proposition}{Proposition}[section]
\newtheorem{corollary}{Corollary}[section]

\newtheorem{definition}{Definition}[section]
\newtheorem{claim}{Claim}[section]
\theoremstyle{remark}

\newtheorem{remark}{Remark}[section]

\newcommand{\CC}{\mathds{C}}
\newcommand{\RR}{\mathds{R}}
\newcommand{\NN}{\mathds{N}}
\newcommand{\HH}{\mathds{H}}

\newcommand{\ind}{\mathds{1}}

\newcommand{\VV}{\mathbb{V}}
\newcommand{\EE}{\mathbb{E}}
\newcommand{\PP}{\mathbb{P}}

\begin{document}
\title{Partial sums of biased random multiplicative functions}
\author{M. Aymone }
\author{V. Sidoravicius }
\subjclass[2010]{Primary: 11N37. Secondary: 60F15.}
\keywords{Random multiplicative functions; Probabilistic Number Theory; Riemann Hypothesis}

\maketitle
\begin{abstract}
Let $\mathcal{P}$  be the set of the primes.  We consider a class of random multiplicative functions $f$ supported on the squarefree integers, such that $\{f(p)\}_{p\in\mathcal{P}}$ form a sequence of $\pm1$ valued independent random variables with $\mathbb{E} f(p)<0$,  $\forall p\in \mathcal{P}$.  The function $f$ is called strongly biased (towards classical M\"obius function), if  $\sum_{p\in\mathcal{P}}\frac{f(p)}{p}=-\infty$ \textit{a.s.}, and it is weakly biased if  $\sum_{p\in\mathcal{P}}\frac{f(p)}{p} $ converges \textit{a.s.} Let $M_f(x):=\sum_{n\leq x}f(n)$.   We establish a number of necessary and sufficient conditions for $M_f(x)=o(x^{1-\alpha})$ for some $\alpha>0$, \textit{a.s.},
when $f$ is  strongly or  weakly biased,
and prove that the Riemann Hypothesis holds if and only if $M_{f_\alpha}(x)=o(x^{1/2+\epsilon})$ for all $\epsilon>0$ \textit{a.s.}, for each $\alpha>0$, where $\{f_\alpha \}_\alpha$ is a certain family of weakly biased random  multiplicative functions.
\end{abstract}

\section{Introduction.}
A function $f:\NN\to\CC$ is called multiplicative function if $f(1) = 1$ and $f(nm)=f(n)f(m)$ whenever $n$ and $m$ are coprime. Let $\mathcal{P}$ be the set of the prime numbers.  In this paper we consider a class of multiplicative functions $f$ which are  supported on the {square-free} integers, {\emph{i.e.}} $f(n)=0$ for all $n\in\NN$, for which $\exists \; p\in\mathcal{P}$ such that $p^2| n$. A function $f$ from this class is called random (binary) multiplicative function if $\{f(p)\}_{p\in\mathcal{P}}$ form a sequence of $\pm1$ valued independent random variables.

Let $\mu$ be the M\"obius function, the multiplicative function supported on the square-free integers with $\mu(p)=-1$ $\forall p\in\mathcal{P}$. We say that $f$ is {\it{biased}} (towards $\mu$) if $\EE f(p)<0$ $\forall p\in\mathcal{P}$.
If $f$ is biased and $\sum_{p\in\mathcal{P}}\frac{f(p)}{p}$ converges \textit{a.s.}, we say that $f$ is {\it{weakly biased}};
otherwise, if $\sum_{p\in\mathcal{P}}\frac{f(p)}{p}=-\infty$ \textit{a.s.}, we say that $f$ is {\it{strongly biased}}. In the case $(f(p))_{p\in\mathcal{P}}$ is i.i.d. with $\EE f(2)=0$, we say that $f$ is an unbiased binary random multiplicative function.

Further, for $x\geq 1$, we denote $M_f(x):=\sum_{n\leq x}f(n)$.

A classical result of J.E.Littlewood, \cite{little}, states that the Riemann Hypothesis (RH) holds if and only if the Merten's function $M_\mu(x)=o(x^{1/2+\epsilon})$, $\forall \epsilon>0$. This criterion led A. Wintner to investigate what happens with the partial sums $M_w(x)$ of a random multiplicative function $w$. In \cite{wintner}, A. Wintner proved that if $w$ is unbiased, then $M_w(x)=o(x^{1/2+\epsilon})$, $\forall \epsilon>0$ \textit{a.s.} Since then, many results pursuing the exact order of $M_w(x)$ have been proved \cite{erdosuns,halasz,basquinn,harpergaussian,tenenbaum2013}, and also Central Limit Theorems have been established \cite{houghmult,harpermult,chatterjeemult}.

These results naturally raises a question what can be said for $M_f(x)$ in the case of biased $f$. To begin with, let $f$ be strongly biased and such that the series $\sum_{p\in\mathcal{P}}(1+\EE f(p))$ converges. Then, by the Borel-Cantelli Lemma, the random subset of primes $\{p\in\mathcal{P}: f(p)\neq \mu(p)\}$ is finite \textit{a.s.} In particular, $M_\mu(x)$ and $M_f(x)$ have essentially the same asymptotic behavior. Hence, in this case, by J.E. Littlewood's criterion, we can reformulate RH in terms of the asymptotic behavior of $M_f(x)$. On the other hand, a similar argument shows that if $f$ is weakly biased and such that $\sum_{p\in\mathcal{P}} \EE f(p)$ converges, then $f$ is essentially the unbiased $w$ and hence, as a consequence from A. Wintner's Theorem, $M_f(x)=o(x^{1/2+\epsilon})$ for all $\epsilon>0$ \textit{a.s.}

In this paper we are interested in determining the range of biased $f$ such that we can reformulate RH in terms of the asymptotic behavior of $M_f(x)$, in the case that
the random subset of primes $\{p\in\mathcal{P}: f(p)\neq \mu(p)\}$ is infinite \textit{a.s.} Also, we are interested in determining the range of biased $f$ such that $M_f(x)=o(x^{1/2+\epsilon})$ $\forall \epsilon>0$ \textit{a.s.}

In the case that $f$ is weakly biased, our first result states:
\begin{theorem}\label{alphaa} Let $\alpha>0$ and $f_\alpha$ is such that $\EE f_\alpha(p)=-\frac{1}{p^\alpha}$ $\forall p\in\mathcal{P}$. Then the Riemann hypothesis holds if and only if $M_{f_\alpha}(x)=o(x^{1/2+\epsilon})$ for all $\epsilon>0$ a.s., for each $\alpha>0$.
\end{theorem}
Further, for $f_\alpha$ as in Theorem \ref{alphaa}, the Dirichlet series $F_\alpha(z)=\sum_{n=1}^\infty\frac{f_\alpha(n)}{n^z}$ converges in the half plane $Re(z)>1-\alpha$ \textit{a.s.} In particular, $M_{f_\alpha}(x)=o(x^{1-\alpha+\epsilon})$, $\forall \epsilon>0$ \textit{a.s.} This convergence can be seen as a consequence from the fact that $f_\alpha$ has a small bias (see Theorem \ref{viz1/22} below), and naturally raises the question of whether $F_\alpha(z)$ converges at $z=1-\alpha$ \textit{a.s.} We point that the convergence of $F_\alpha(1-\alpha)$ \textit{a.s.} is necessary for RH, and if it converges, then its probability law is a Dirac measure at $0$, which reveals an interesting dependence relation among the random variables $\{f_\alpha(n)\}_{n\in\NN}$. We established this convergence only when  $0<\alpha<1/3$, although it holds in probability for any $0<\alpha<1/2$. The reason $<1/3$ is related to the Vinogradov-Korobov zero free region for $\zeta$ (the best up to date), combined with certain probabilistic bounds which seems to be optimal (see Proposition \ref{vino}).

In the case that $f$ has a strong bias, we have:
\begin{theorem}\label{alphaa2}
Let $f$ be strongly biased such that, for some fixed $0<\alpha\leq 1/2$, the series $\sum_{p\in\mathcal{P}}\frac{1+f(p)}{p^{1-\alpha+\epsilon}}$ converges $\forall\epsilon>0$ a.s. Then $M_f(x)=o(x^{1-\alpha+\epsilon})$, $\forall\epsilon>0$ a.s. if and only if $M_\mu(x)=o(x^{1-\alpha+\epsilon})$, $\forall \epsilon>0$.
\end{theorem}

The next question concerns necessary and sufficient conditions on biased $\{f(p)\}_{p\in \mathcal{P}}$, under which $M_f(x)=o(x^{1-\delta})$ for some possibly random $0<\delta<1/2$ \textit{a.s.} Let $F(z):=\sum_{n=1}^\infty \frac{f(n)}{n^z}$, $z\in\CC$, $Re(z)>1$, be the Dirichlet series of $f:\NN\to\CC$. In \cite{koukou}  this problem has been studied in a more general context, where a multiplicative function $f$ may assume values on $\mathbb{U}=\{z\in\CC:|z|\leq 1\}$, and $\{f(p)\}_{p\in\mathcal{P}}$ is not necessarily a random sequence.  In particular, for completely multiplicative functions $f:\NN\to[-1,1]$ Theorem 1.6 of \cite{koukou}, states, that if for some $\delta\in(0,1/3)$ and $Q\geq \exp(1/\delta)$ one has $|M_f(x)|\leq \frac{x^{1-\delta}}{(\log x)^2}$ $\forall x\geq Q$, then there exists $c=c(\delta)$ and $d=d(f)$ such that
\begin{align*}
\sum_{p\leq x}f(p)\log p &\ll \frac{x}{\exp(c\sqrt{\log x})} + x^{1-cd}, \quad \mbox{ if }F(1)\neq 0
\\
\sum_{p\leq x}(1+f(p))\log p &\ll x^{1-\frac{1}{61\log Q}}, \qquad \qquad \qquad \; \mbox{ if }F(1)=0,
\end{align*}
and it is also applicable to biased random multiplicative functions, by identifying the condition
$F(1)\neq 0$ \textit{a.s.} with weakly biased $f$, and $F(1)=0$ \textit{a.s.} with strongly biased $f$. For the general account on the state of the art we refer reader to  \cite{koukou}, and the references therein, and also to
\cite{tenenbaumlivro}, Chapters II.5 and III.4, and their historical notes.

For weakly biased $f$, we prove:
\begin{theorem}\label{viz1/2} Let $f$ be weakly  biased. If $M_f(x)=o(x^{1-\delta})$ for some possibly random $0<\delta<1/2$ a.s., then there exists $0<\alpha<1/2$ such that the random series $\sum_{p\in\mathcal{P}}\frac{f(p)}{p^{1-\alpha}}$ converges a.s.
\end{theorem}
\noindent Hence, if $f$ is as in Theorem \ref{viz1/2}, then there exists $\alpha>0$ such that for each $\epsilon>0$,
$$\sum_{p\leq x}f(p)\log p \ll x^{1-\alpha+\epsilon},\;\quad a.s.$$
In particular, if for each $p\in\mathcal{P}$ we have $\EE f(p)=-\frac{1}{\exp(\sqrt{\log p})}$, then $M_f(x)$ is not $o(x^{1-\delta})$ for any $\delta>0$, \textit{a.s.}

For $f$ strongly biased we prove the following:
\begin{theorem}\label{alphaa1} Let $f$ be strongly biased. If for some fixed $0<\alpha<1/2$, $M_f(x)=o(x^{1-\alpha})$ a.s., then the series $\sum_{p\in\mathcal{P}}\frac{1+f(p)}{p^{1-\alpha+\epsilon}}$ converges $\forall\epsilon>0$ a.s.
\end{theorem}

Observe that, for fixed $0<\alpha\leq 1/2$ and $f$ strongly biased, if $M_f(x)=o(x^{1-\alpha+\epsilon})$ $\forall\epsilon>0$ \textit{a.s.}, then by Theorem \ref{alphaa1}, $\sum_{p\in\mathcal{P}}\frac{1+f(p)}{p^{1-\alpha+\epsilon}}$ converges $\forall\epsilon>0$ \textit{a.s.} Hence, by Theorem \ref{alphaa2}, $M_\mu(x)=o(x^{1-\alpha+\epsilon})$ $\forall\epsilon>0$, and this implies that the Riemann zeta function $\zeta$ has no zeroes in $\{z\in\CC:Re(z)>1-\alpha\}$. Thus, in the case of strongly biased $f$, in order to provide conditions that guarantee $M_f(x)=o(x^{1-\epsilon})$ for some $\epsilon>0$ \textit{a.s.}, we must assume certain half planes to be zero free regions of $\zeta$.

Let $\ast$ denote the Dirichlet convolution. When $f$ is weakly biased, $f$ can be represented as $f=w\ast g$ (see Remark \ref{uniform coupling} and Claim \ref{series}), where $w$ and $g$ are random multiplicative functions which possibly admit zero values on primes,  $w$ is unbiased, and $g$ is such that
$\EE g(p) = \EE f(p), \, \forall p \in \mathcal{P}$. Since $M_w(x)=o(x^{1/2+\epsilon}), \; \forall \epsilon > 0$ \textit{a.s}, in contrast with the class of strongly biased random multiplicative functions, allows us to derive conditions which do not depend on zero free regions of $\zeta$, and which guarantee that for weakly biased $f$ we get that $M_f(x)=o(x^{1-\alpha})$ for some $\alpha>0$ \textit{a.s.}
\begin{theorem}\label{viz1/22} Let $f$ be weakly biased, such that for some fixed $0<\alpha<1/2$, $\EE f(p)=-\frac{\delta_p}{p^\alpha}$, where $0\leq \delta_p \leq 1$. Then $M_f(x)=o(x^{1-\alpha+\epsilon})$ for all $\epsilon>0$, a.s. If in addition we assume $\limsup{\delta_p}<1$ and $\sum_{p\in\mathcal{P}} \frac{\delta_p}{p}=\infty$, then $M_f(x)$ is not $o(x^{1-\alpha-\epsilon})$ for any $\epsilon>0$, \textit{a.s.}\\
\end{theorem}

Let $g,h:\NN\to \{z\in\CC: |z|\leq 1\}$ be multiplicative functions. For $x\geq 1$ consider the function
$$\mathbb{D}(g,h,x)^2:=\sum_{p\leq x} \frac{1-Re (g(p)\overline{h}(p))}{p}.$$
The function $\mathbb{D}$ measures the ``distance'' up to $x$ between multiplicative functions, and we informally say that
$h$ pretends to be $g$ if $\mathbb{D}(g,h,x)$ is small compared with $x$. We refer reader to \cite{sound} for the general account and interesting results regarding this function.

Let $f$ and $g$ be binary random multiplicative functions. We say that $f$ and $g$ are uniformly coupled if for each $p\in\mathcal{P}$ we have $\PP(f(p)\neq g(p))=|\PP(f(p)=-1)-\PP(g(p)=-1)|$. For a construction of this coupling we refer reader to Remark \ref{uniform coupling}.
For such $f,g$ and $1/2<\sigma\leq 1$ we define
$$\mathbb{D}_\sigma(f,g)^2:=\sum_{p\in\mathcal{P}} \frac{1-f(p)g(p)}{p^\sigma}.$$
Note that $\mathbb{D}_\sigma(f,g)$ may be $\infty$ and that $\mathbb{D}_{\sigma_1}(f,g)\leq \mathbb{D}_{\sigma_2}(f,g)$ whenever $\sigma_2\leq\sigma_1$. Moreover,
$$\EE \mathbb{D}_\sigma(f,g)^2=\sum_{p\in\mathcal{P}}\frac{2\PP(f(p)\neq g(p))}{p^\sigma},$$
and hence, by the Kolmogorov two series Theorem, $\mathbb{D}_\sigma(f,g)<\infty$ \textit{a.s.} if and only if $\sigma>1/2$ and satisfies $\EE \mathbb{D}_\sigma(f,g)^2<\infty$. In particular, if $w$ is unbiased and $f_\alpha$ is as in Theorem \ref{alphaa}, and if $w$ and $f_\alpha$ are uniformly coupled, then for each $\sigma>1-\alpha$, $\mathbb{D}_\sigma(w,f_\alpha)<\infty$ \textit{a.s.}

In this point of view, if we assume that $f$ is strongly biased, and if for some $0<\alpha<1/2$ we have $M_f(x)=o(x^{1-\alpha})$ \textit{a.s.}, by Theorem \ref{alphaa1}, for each $\sigma>1-\alpha$ we obtain that  $\mathbb{D}_\sigma(f,\mu)<\infty$ \textit{a.s.} Further, for  strongly biased $f$ such that for each $\sigma>1-\alpha$,   $\mathbb{D}_\sigma(f,\mu)<\infty$  \textit{a.s.}, by Theorem \ref{alphaa2}, $M_\mu(x)=o(x^{1-\alpha+\epsilon})$ $\forall\epsilon>0$ if and only if $M_f(x)=o(x^{1-\alpha+\epsilon})$ $\forall\epsilon>0$ \textit{a.s.}

On the other hand, let $f$ be weakly biased and $w$ unbiased, and assume that $f$ and $w$ are uniformly coupled. By Theorem \ref{viz1/2}, if we assume that for some fixed small $\delta>0$ we have $M_f(x)=o(x^{1-\delta})$ \textit{a.s.}, then we obtain $0<\alpha<1/2$ such that for all $\sigma>1-\alpha$, $\mathbb{D}_\sigma(w,f)<\infty$ \textit{a.s.}
Further, the half plane $Re(z)>1-\alpha$ is a zero free region for the corresponding Dirichlet series of $f$, \textit{a.s.} This naturally raises a question  what can be said for the class of weakly biased $f$ such that, for some fixed $0<\alpha<1/2$, $F(z)=\sum_{n=1}^\infty \frac{f(n)}{n^z}$ converges in the half plane $Re(z)>1-\alpha-\eta$, for some possibly random $\eta>0$ \textit{a.s.}, and in the closed  half plane $Re(z)\geq 1-\alpha$, $F$ has at most a single zero, and possibly multiple, at $\beta=1-\alpha$, \textit{a.s.}

In the case that $F$ has at most a single and simple zero at $\beta$ \textit{a.s.}, we prove the following result:
\begin{theorem}\label{pretentious} Let $f$ be weakly  biased such that for some fixed $0<\alpha<1/2$, for each prime $p$,
$\EE f(p)=-\frac{\delta_p}{p^\alpha}$, where $0\leq \delta_p\leq 1$.  Let $w$ be unbiased and $f_\alpha$ as in Theorem \ref{alphaa}, and assume that $w,f_\alpha$ and $f$ are uniformly coupled. If $M_f(x)=o(x^{1-\alpha-\eta})$ for some possibly random $\eta>0$ a.s., then there exists $\epsilon>0$ such that:\\
\noindent a) $\mathbb{D}_{1-\alpha-\epsilon}(f,w)\;\,<\infty$ \textit{a.s.}, if $\sum_{p\in \mathcal{P}}\frac{\EE f(p)}{p^{1-\alpha}}>-\infty$;\\
\noindent b) $\mathbb{D}_{1-\alpha-\epsilon}(f,f_\alpha)<\infty$ \textit{a.s.}, if $\sum_{p\in \mathcal{P}}\frac{\EE f(p)}{p^{1-\alpha}}=-\infty$. Further, in this case, $\zeta$ has no zeros in the half plane $Re(z)>1-\alpha-\epsilon$.
\end{theorem}
In particular, if for all $p\in\mathcal{P}$ we have $\delta_p=\frac{1}{\exp(\sqrt{\log p})}$, or if for all $p\in\mathcal{P}$ we have $\delta_p=1-\frac{1}{\exp(\sqrt{\log p})}$, then by Theorem \ref{pretentious}, $M_f(x)$ is not $o(x^{1-\alpha-\epsilon})$ for any $\epsilon>0$ \textit{a.s.}

We conclude by mentioning that, for each fixed $k\in\NN$, a similar result to Theorem \ref{pretentious} b) holds, in the case that $F(z)=\sum_{n=1}^\infty\frac{f(n)}{n^z}$ has a single zero at $\beta=1-\alpha$, of multiplicity $k$, and $\delta_p$ assume values in the interval $[0,k]$, see Theorem \ref{pretentiousk}.

\bigskip

The paper is organized as follows. In Section \ref{definicoes} we set up the main notations and tools from Probability and Analytic Number Theory. In Section \ref{random dirichlet} we consider the problem of bounding convergent random Dirichlet Series $\sum_{n=1}^\infty \frac{X_n}{n^z}$ ($Re(z)>1/2$) in vertical strips, where $\{X_n\}_{n\in\NN}$ belongs to a certain class of sequences of random variables. In Section \ref{provas} we prove all the main results and in Section \ref{concluding} we conclude with some remarks.




\section{Preliminaries.}\label{definicoes}
\noindent \textit{Notations from Probability Theory.} $(\Omega,\mathcal{F},\PP)$  stands for a probability space. Given a set $E\in\mathcal{F}$, the random variable $\ind_E:\Omega\to\{0,1\}$ stands for the indicator function of $E$, that is, $\ind_E(\omega)=1$ if $\omega\in E$ and $\ind_E(\omega)=0$ otherwise. Given an square integrable random variable $Y:\Omega\to\RR$:
\begin{align*}
\EE Y:=&\int_{\Omega}Y(\omega)\PP(d\omega),\\
\VV Y:=&\EE Y^2- (\EE Y)^2.
\end{align*}
\textit{Notations from Complex Analysis.} A set of the form $\HH_a:=\{z\in\CC:\,Re(z)>a\}$ where $a\in\RR$ is called half plane. Let $R_1\subset R_2$ be two open connected sets of $\CC$ and $h:R_1\to\CC$ be an analytic function. We say that $h$ has analytic extension to $R_2$ if there exists an analytic function $\overline{h}:R_2\to\CC$ such that for all $z\in R_1$ we have that $\overline{h}(z)=h(z)$.
\begin{definition} Let $S\subset\CC$. A map $f:S\times\Omega\to\CC$ is called a random function if $\omega\in \Omega\mapsto f(s,\omega)$ is a complex valued random variable for each  fixed $s\in S$, and $s\in S \mapsto f(s,\omega)$ is a function of one complex variable for each fixed $\omega\in\Omega$.
\end{definition}
\noindent Let $f:S\times\Omega\to\CC$ be a random function. For each fixed $\omega\in\Omega$, $f_\omega$ denotes the function $f_\omega:S\to\CC$ given by $f_\omega(s):=f(s,\omega)$.
\begin{definition} Let $S\subset\CC$ be an open connected set and $f:S\times\Omega\to\CC$ a random function. We say that $f$ is a random analytic function if the set of elements $\omega\in\Omega$, for which $f_\omega:S\to\CC$ is analytic, contains  a set $\Omega^*\in\mathcal{F}$ such that $\PP(\Omega^*)=1$.
\end{definition}
\noindent Let $(X_k)_{k\in\NN}$ be a sequence of independent random variables such that $\VV X_k^2<\infty$ for all $k$. We define
\begin{align}
\sigma_1=&\inf\bigg{\{}0<\sigma\leq\infty:\mbox{ the series }\sum_{k=1}^\infty \frac{\EE X_k}{k^\sigma}\mbox{ converges } \bigg{\}},\label{sigma1}\\
\sigma_2=&\inf\bigg{\{}0<\sigma\leq\infty:\mbox{ the series }\sum_{k=1}^\infty \frac{\VV X_k}{k^{2\sigma}}\mbox{ converges }\bigg{\}}\label{sigma2}.
\end{align}
\begin{proposition}\label{kolmogorovlandau} Let $(X_k)_{k\in\NN}$ be a sequence of independent random variables and $\sigma_1$ and $\sigma_2$ be as in $(\ref{sigma1})$ and $(\ref{sigma2})$. Assume that $\sigma_c=\max\{\sigma_1,\sigma_2\}<\infty$. Then $F:\HH_{\sigma_c}\times\Omega\to\CC$ given by $F(z):=\sum_{k=1}^\infty\frac{X_k}{k^z}$ converges for each $z\in\HH_{\sigma_c}$ and it is a random analytic function.
\end{proposition}
\begin{proof}
Let $\{c_k\}_{k=1}^\infty$ be a sequence of complex numbers and $\sum_{k=1}^\infty \frac{c_k}{k^z}$ be its Dirichlet series, where $z\in\CC$. A classical result in the Theory of the Dirichlet series (see \cite{apostol}, Theorems 11.8 and 11.11) states that if the series $\sum_{k=1}^\infty \frac{c_k}{k^{z}}$ converges for $z_0=\sigma_0+it_0$ then it converges for all $z\in\HH_{\sigma_0}$ and also uniformly on compact subsets of this half plane. Thus the function $z\in\HH_{\sigma_0}\mapsto\sum_{k=1}^\infty \frac{c_k}{k^z}$ is analytic. The Kolmogorov two series Theorem states that if $\{Y_k\}_{k=1}^\infty$ is a sequence of independent random variables such that $\sum_{k=1}^\infty \EE Y_k$ and $\sum_{k=1}^\infty \VV Y_k$ converge then $\sum_{k=1}^\infty Y_k$ converges \textit{a.s.} Thus the assumption that $\sigma_c<\infty$ implies that for each $\sigma>\sigma_c$ both series $\sum_{k=1}^\infty \frac{\EE X_k}{k^\sigma}$ and $\sum_{k=1}^\infty \frac{\VV X_k}{k^{2\sigma}}$ converge. Hence by the Kolmogorov two series Theorem, for each $j\in\NN$ and $\sigma_j=\sigma_c+j^{-1}$, each event $\Omega_j:=[F(\sigma_j)\mbox{ converges }]$ has $\PP(\Omega_j)=1$ so as $\Omega^*:=\bigcap_{j=1}^\infty\Omega_j$. By the referred properties of convergence of a Dirichlet series, we obtain that for each $\omega\in\Omega^*$ the Dirichlet series $F_\omega(z)$ converges for each $z\in\HH_{\sigma_c}$ and uniformly in compact subsets of this half plane. We then conclude that $F$ is a random analytic function. \end{proof}
\noindent \textit{Notations from Number Theory.} In the sequel $\mathcal{P}$ stands for the set of the prime numbers and $p$ for a generic element of $\mathcal{P}$. Given $d,n\in\NN$, $d|n$ and $d\nmid n$ means that $d$ divides and that $d$ do not divides $n$, respectively. The M\"obius function is denoted by $\mu$ and its partial sums by $M_\mu(x):=\sum_{k\leq x}\mu(k)$.
\begin{definition} A random function $f:\NN\times\Omega\to\CC$ is called random multiplicative function if $f(1)=1$,
\begin{equation}\label{multiplicative}
f(n)=|\mu(n)|\prod_{p|n}f(p)\quad(n\geq 2),
\end{equation}
and $\{f(p)\}_{p\in\mathcal{P}}$ is a sequence of $\pm 1$ independent random variables.
\end{definition}
\begin{lemma}\label{lema0} Let $f$ be a random multiplicative function and for each $z\in\HH_1$ let $F(z):=\sum_{k=1}^\infty \frac{f(k)}{k^z}$. Then:\\
i) $\EE F(z):\HH_1\to\CC$ is analytic and $F:\HH_1\times\Omega$ is a random analytic function. Moreover for all $z\in\HH_1$, $\EE F(z)\neq0$ and $F_\omega(z)\neq0$ for each $\omega\in\Omega$.\\
ii) There exists a non-vanishing random analytic function $\theta:\HH_{1/2}\times\Omega\to\CC$ such that for each $z\in\HH_1$, $\theta(z)=\frac{F(z)}{\EE F(z)}$.\\
iii) The random analytic function $\theta$ is given by
\begin{equation}\label{thetta}
\theta(z)=\exp\bigg{(}\sum_{p\in\mathcal{P}}\frac{f(p)-\EE f(p)}{p^z}\bigg{)}\exp(A(z)),
\end{equation}
where $A:\HH_{1/2}\times\Omega\to\CC$ is a random analytic function such that for all $\sigma\geq\sigma_0>\frac{1}{2}$ there exists $C=C(\sigma_0)$ such that $|A(\sigma+it)|\leq C$ \textit{a.s.}\\
\end{lemma}
\begin{proof}  Let $h:\NN\to[-1,1]$ be a multiplicative function supported on the square free integers. Let $g_k:\CC\to\CC$ be given by $g_k(z)=\frac{h(k)}{k^z}$. Then for each $k\in\NN$, $g_k$ is analytic and satisfies $|g_k(z)|\leq\frac{1}{k^\sigma}$ where $\sigma=Re(z)$. Thus for each $\sigma>1$, $\sum_{k=1}^\infty g_k(z)$ is a series of complex analytic functions that converges uniformly in the set $\{z\in\CC:Re(z)\geq \sigma\}$ and hence uniformly on compact subsets of $\HH_1$. This gives that the Dirichlet series $z\in\HH_1\mapsto\sum_{k=1}^\infty \frac{h(k)}{k^z}$ is analytic. The same argument gives that $z\in\HH_1\mapsto\sum_{p\in\mathcal{P}} \frac{h(p)}{p^z}$ is analytic.
\begin{claim}\label{cl1} Let $h$ be as above. Then for each $z\in\HH_1$
$$\sum_{k=1}^\infty \frac{h(k)}{k^z}=\exp\bigg{(}\sum_{p\in\mathcal{P}}\frac{h(p)}{p^z}\bigg{)}\exp(A(h,z)),$$
where $z\in\HH_{1/2}\mapsto A(h,z)$ is analytic and uniformly bounded in the set $\{z\in\CC:Re(z)\geq \sigma_0\}$ for each real $\sigma_0>1/2$.
\end{claim}
\noindent \textit{Proof of the claim.} The Dirichlet series $\sum_{k=1}^\infty \frac{h(k)}{k^z}$ has Euler product representation (see \cite{apostol}, Theorem 11.6): For each $z\in\HH_1$
\begin{equation}\label{pr1}
\sum_{k=1}^\infty \frac{h(k)}{k^z}=\prod_{p\in\mathcal{P}}\bigg{(}1+\frac{h(p)}{p^z}\bigg{)}.
\end{equation}
Since the Taylor series $\log(1+x)=\sum_{m=1}^\infty \frac{(-1)^{m+1}}{m}x^m$ converges absolutely for $|x|<1$, we obtain for each real $\sigma>1$ that $\log\big{(}1+\frac{h(p)}{p^\sigma}\big{)}=\frac{h(p)}{p^\sigma}+A_p(\sigma)$ where $A_p(\sigma):=\sum_{m=2}^\infty \frac{(-1)^{m+1}}{m}\frac{h(p)^\sigma}{p^{m\sigma}}$. Let $z\in\CC$ be such that $Re(z)=\epsilon>0$. Observe that for large $p$
$$|A_p(z)|\leq \sum_{m=2}^\infty \frac{1}{p^{m\epsilon}}=\frac{1}{p^{\epsilon}(p^{\epsilon}-1)}\sim \frac{1}{p^{2\epsilon}}.$$
Hence $A_p(z)$ is analytic in $\HH_0$ and there is $C>0$ such that $|A_p(z)|\leq \frac{C}{p^{2\epsilon}}$ $\forall p\in\mathcal{P}$ ($Re(z)=\epsilon$). Since for $\sigma>{1/2}$ the series $\sum_{p\in\mathcal{P}}\frac{1}{p^{2\sigma}}$ is summable, the series of complex analytic functions $A(h,z):=\sum_{p\in\mathcal{P}}A_p(z)$ converges absolutely and uniformly in the set $\{z\in\CC:\Re(z)\geq \sigma\}$ for each $\sigma>1/2$, and hence uniformly on compact subsets of $\HH_{1/2}$. This gives that $A(h,z)$ is analytic in $\HH_{1/2}$ and for each $\sigma_0>1/2$ it is uniformly bounded by some constant $C=C(\sigma_0)$ in the set $\{z\in\CC:Re(z)\geq\sigma_0\}$. This gives the desired properties for $z\in\HH_{1/2}\mapsto A(h,z)$ and the following formula for each $\sigma>1$:
$$\log\prod_{p\in\mathcal{P}}\bigg{(}1+\frac{h(p)}{p^\sigma}\bigg{)}=\sum_{p\in\mathcal{P}}\frac{h(p)}{p^\sigma}+A(h,\sigma).$$
This formula combined with $(\ref{pr1})$, gives for each $\sigma>1$:
\begin{equation}\label{cl11}
\sum_{k=1}^\infty\frac{h(k)}{k^\sigma}=\exp\bigg{(}\sum_{p\in\mathcal{P}}\frac{h(p)}{p^\sigma}\bigg{)}\exp(A(h,\sigma)).
\end{equation}
Let $F_1$ and  $F_2$ be two complex analytic functions defined in some open connected set $U\subset\CC$ such that $F_1(z_k)=F_2(z_k)$ where $\{z_k\}_{k=1}^\infty\subset U$ is a convergent sequence whose limit point is $z\in U$. Then $F_1=F_2$ (see \cite{conway}, Corollary 3.8 and 3.9). This gives that $(\ref{cl11})$ holds for all $z\in\HH_1$, since the left side and the right side of this equation are the restriction of complex analytic functions to the set $\{\sigma\in\RR:\sigma>1\}$, finishing the proof of the claim.

\noindent \textit{Proof of }\textit{i}) For each $\omega\in\Omega$, $n\mapsto f_\omega(n)$ is a multiplicative function supported on the square free integers. Since $\{f(p)\}_{p\in\mathcal{P}}$ is a sequence of independent random variables, $n\mapsto \EE f(n)$ also is a multiplicative function supported on the square free integers. Thus, by claim \ref{cl1}, for each $\omega\in\Omega$, $F_\omega:\HH_1\to\CC$ and $\EE F:\HH_1\to\CC$ are non-vanishing complex analytic functions, completing the proof of \textit{i}.

\noindent \textit{Proof of }\textit{ii}) and \textit{iii}) Claim \ref{cl1} gives the following formula for each $\omega\in\Omega$ and $z\in\HH_1$:
$$\frac{F_\omega(z)}{\EE F(z) }=\exp\bigg{(}\sum_{p\in\mathcal{P}}\frac{f_\omega(p)-\EE f(p)}{p^z}\bigg{)}\exp(A(f_\omega,z)-A(\EE f,z)),$$
where $z\in\HH_{1/2}\mapsto A(f_\omega,z)$ and $z\in\HH_{1/2}\mapsto A(\EE f,z)$ are complex analytic functions which are uniformly bounded in the sets $\{z\in\CC:Re(z)\geq \sigma\}$, for each $\sigma>1/2$. Hence $A:\HH_{1/2}\times\Omega\to\CC$ given by $A_\omega(z)=A(f_\omega,z)-A(\EE f,z)$ is the desired random analytic function of iii. By Proposition \ref{kolmogorovlandau}, $z\in\HH_{1/2}\mapsto \sum_{p\in\mathcal{P}}\frac{f(p)-\EE f(p)}{p^z}$ is a random analytic function and hence its exponential also is. This gives the desired properties of the random analytic function $\theta$. \end{proof}
\section{Bounding random Dirichlet series in vertical strips.}\label{random dirichlet}
Let $F(z)=\sum_{n=1}^\infty \frac{f(k)}{k^z}$ ($z\in\CC$) be a general Dirichlet series, and let $\sigma_c\leq\sigma_a<\infty$ be its abscissas of conditional and of absolute convergence respectively. A classical result (see \cite{tenenbaumlivro} page 119, Theorem 15) states that if $F$ has finite $\sigma_c$, then for each $\sigma_0>\sigma_c$ and $\epsilon>0$, uniformly for $\sigma_0\leq\sigma\leq\sigma_c+1$ we have that $F(\sigma+it)\ll |t|^{1-(\sigma-\sigma_c)+\epsilon}$.

In this section we consider the problem of bounding a random Dirichlet series $\sum_{k=1}^\infty\frac{X_k}{k^{\sigma+it}}$ as $t\to\infty$ with fixed $\sigma>1/2$, where $\{X_k\}_{k\in\NN}$ are centered random variables not necessarily independent. If $\{X_k\}_{k\in\NN}$ is a sequence of independent and identically distributed random variables such that $\PP(X_1=-1)=\PP(X_1=1)=\frac{1}{2}$, in \cite{carlson} F.Carlson proved that, for suitable $\{a_k\}_{k\in\NN}\subset \CC$, for each $\sigma>1/2$ we have that $\sum_{k=1}^\infty \frac{a_kX_k}{k^{\sigma+it}}=o(\sqrt{\log t})$ \textit{a.s.} Following the same line of reasoning we prove the following result:
\begin{theorem}\label{marcin} Let $\{X_k\}_{k\in\NN}$ be a sequence of centered and uniformly bounded random variables. Denote for a complex $z$, $F(z):=\sum_{k=1}^\infty \frac{X_k}{k^z}$. Let $1/2<\sigma_0\leq 1$. If $\{X_k\}_{k\in\NN}$ are either a) independent, b) a martingale difference or c) $\rho^*$-mixing\footnote{For the definition of $\rho^*$-mixing see the proof of Theorem \ref{marcin} below.}, then uniformly for all $\sigma_0\leq\sigma\leq1$:
\begin{equation}\label{rds}
F(\sigma+it)\ll (\log t)^{\vartheta(1-\sigma)} \log\log t,\quad a.s.,
\end{equation}
where $\vartheta=1$ in the case a), $\vartheta=2$ in the case b) and $\vartheta=3$ in the case c). Moreover, if $X_k=0$ for all non prime $k$ then the term
$\log \log t $ in $(\ref{rds})$ can be substituted by $\log\log\log t$.
\end{theorem}
In particular, if $\{X_k-\EE X_k\}_{k\in\NN}$ is as above, and the series $\EE F(\sigma)=\sum_{k=1}^\infty \frac{\EE X_k}{k^{\sigma}}$ converges for each $\sigma>1/2$, then
the same statement holds if we substitute $F(\sigma+it)$ by $F(\sigma+it)-\EE F(\sigma+it)$ in the left side of (\ref{rds}).

\noindent Next we will start the proof of Theorem \ref{marcin}. We restrict ourselves to sequences $\{X_k\}_{k\in\NN}$ which satisfy the following conditions:\\
\textit{i}) For all $k$, $\EE X_k=0$ and $|X_k|\leq C$ for some constant $C>0$; \\
\textit{ii}) The random series $\sum_{k=1}^\infty \frac{X_k}{k^z}$ converges for all $z\in\HH_{1/2}$, \textit{a.s.};\\
\textit{iii}) There exists a constant $\gamma>0$ and a increasing function $\lambda:[0,\infty)\to [1,\infty)$ such that $\lim_{t\to\infty}\lambda(t)=\infty$, for all $a,b\geq 0$, $\lambda(a+b)\leq e^{\gamma a}\lambda(b)$, and such that the following inequality holds for all $q\geq2$ for all real numbers $\alpha_1,...,\alpha_n$, for each $n\in\NN$:
$$\bigg{\{} \EE \bigg{|} \sum_{k=1}^n \alpha_kX_k \bigg{|}^q \bigg{\}}^{\frac{1}{q}}\leq \lambda(q)\bigg{\{}\EE\bigg{(} \sum_{k=1}^n |\alpha_kX_k|^2 \bigg{)}^{q/2}\bigg{\}}^{\frac{1}{q}}.$$
\begin{lemma}\label{boundrds} Assume that $\{X_k\}_{k\in\NN}$  satisfies conditions i)-iii) above. Let $\psi(t):=\lambda(\log t)$ and $v(\epsilon^{-1})^2:=\sum_{k=1}^\infty \frac{\ind_{[\EE|X_k|>0]}}{k^{1+\epsilon}}.$ Then for each $\sigma>1/2$, uniformly for all $x\in[\sigma,1]$:
$$\sum_{k=1}^\infty\frac{X_k}{k^{x+it}}\ll \psi(t)^{2-2x} v(\log \psi(t))^2,\; a.s.$$
\end{lemma}
\begin{proof}[Proof of Lemma \ref{boundrds}] We begin the proof with the following claim:
\begin{claim}\label{ineq} Let $\{X_k\}_{k\in\NN}$ and $v(\epsilon^{-1})$ be as in Lemma \ref{boundrds}. Then there exists $D>0$ such that  for all $q>1$, $\epsilon>0$ and $t\in\RR$ the following inequality holds:
$$\EE\bigg{|}\sum_{k=1}^\infty \frac{X_k}{k^{\frac{1+\epsilon}{2}+it}}   \bigg{|}^q\leq D^q \lambda(q)^qv(\epsilon^{-1})^q.$$
\end{claim}
\begin{proof}[Proof of claim \ref{ineq}] Let $z=x+iy$ with $x=1/2+\epsilon/2$, where $\epsilon>0$. Denote $|z|=\sqrt{x^2+y^2}$. Since the random series $\sum_{k=1}^\infty\frac{X_k}{k^z}$ converges \textit{a.s.} we obtain by Fatou's Lemma that
$$\EE \bigg{|}\sum_{k=1}^\infty\frac{X_k}{k^z}\bigg{|}^q \leq \liminf_{n\to\infty} \EE \bigg{|}\sum_{k=1}^n\frac{X_k}{k^z}\bigg{|}^q.$$
Recall that $|a_1+ia_2|^q\leq 2^{\frac{q}{2}}(|a_1|^q+|a_2|^q)$ for each $a_1,a_2\in\RR$. Thus taking
$$a_1:=\sum_{k=1}^n\frac{X_k}{k^x}\cos(y\log k),\hspace{0.4cm}a_2:=-\sum_{k=1}^n\frac{X_k}{k^x}\sin(y\log k),$$
we get
$$\EE \bigg{|}\sum_{k=1}^n\frac{X_k}{k^z}\bigg{|}^q\leq 2^{\frac{q}{2}}(\EE|a_1|^q+\EE|a_2|^q).$$
Let $\alpha_k=\frac{\cos(y\log k)}{k^x}$ for all $k\geq 1$. Observe that $\alpha_k^2\leq\frac{1}{k^{1+\epsilon}}$. Hence the condition \textit{iii}) above implies that
$$ \EE |a_1|^q \leq \lambda(q)^q \EE\bigg{(}\sum_{k=1}^n |\alpha_kX_k|^2\bigg{)}^{q/2}\leq \lambda(q)^q\EE\bigg{(}\sum_{k=1}^\infty \frac{C^2\ind_{[\EE |X_k|>0]}}{k^{1+\epsilon}}\bigg{)}^{q/2}\leq C^q \lambda(q)^qv(\epsilon^{-1})^q,$$
where $C>0$ is the constant of condition \textit{i}) above. Similarly we get the same bound for $\EE |a_2|^q$. We complete the proof of the claim by choosing $D=2\sqrt{2}C$. \end{proof}
\medskip
\noindent Let $F:\HH_{1/2}\times\Omega\to\CC$ be given by $F(z):=\sum_{k=1}^\infty \frac{X_k}{k^z}$. Define
$$\Omega^*:=[\omega\in\Omega:F_\omega(z)\mbox{ converges for each } z\in\HH_{1/2}].$$
By \textit{ii)}, $\PP(\Omega^*)=1$. Hence $F$ is a random analytic function (see the proof of Proposition \ref{kolmogorovlandau}). Let $q\geq1$, $\epsilon=\epsilon(q)\in(0,1/3]$, $\sigma=1/2+\epsilon$ and $\sigma'=1/2+\epsilon/2$. Let $R_1$ and $R_2$ be the rectangles:
\begin{align*}
R_1=&R_1(q ,\epsilon)=[\sigma,4/3]\times[-e^{q-2},e^{q-2}],\\
R_2=&R_2(q ,\epsilon,\omega)=[\sigma',\sigma'+1]\times [-\tau'(\omega),\tau(\omega)],
\end{align*}
where $\tau,\tau'\in[e^{q-1},e^{q}]$ will be chosen later. Observe that $R_1\subset R_2$ and the distance from $\partial R_1$ to $\partial R_2$ equals to $\epsilon/2$. Decompose: $\partial R_2=I_1\cup I_2\cup I_3\cup I_4$, where $I_1$ and $I_3$ are the vertical lines at $Re(s)=\sigma'+1$ and $Re(s)=\sigma'$ respectively and $I_2$ and $I_4$ are the horizontal lines at $Im(s)=\tau$ and $Im(s)=-\tau'$ respectively. For $q\in\NN$ and $\omega\in\Omega^*$, define:
\begin{equation}\label{vj}
V_j=V_j(\omega,q)=\int_{I_j}|F_\omega(s)|^q |ds|,\;j=1,2,3,4.\\
\end{equation}

For all $q\in\NN$, $F_\omega^q$ is analytic on $\HH_{1/2}$. Hence, by the Cauchy integral formula, for each $z\in R_1$,
$$F_\omega^q(z)=\frac{1}{2\pi i}\int_{\partial R_2}\frac{F_\omega(s)^q}{s-z}ds=\frac{1}{2\pi i}\sum_{j=1}^4\int_{I_j} \frac{F_\omega(s)^q}{s-z}ds,$$
where for each $j$, the line integral over $I_j$ above is oriented counterclockwise. For fixed $z\in R_1$ and $j=1,2,3,4$:
$$\bigg{|}\frac{1}{2\pi i}\sum_{j=1}^4\int_{I_j} \frac{F_\omega(s)^q}{s-z}ds \bigg{|}\leq \frac{1}{2\pi i}\sum_{j=1}^4\int_{I_j} \frac{|F_\omega(s)^q|}{|s-z|}|ds|\leq \frac{1}{\pi \epsilon }\sum_{j=1}^4V_j(\omega,q).$$
(see \cite{conway}, p. 65). Hence:
\begin{equation}\label{aux60}
\max_{z\in R_1}|F_\omega(z)|^q \leq \frac{1}{\pi \epsilon}\sum_{j=1}^4 V_j(q,\omega).
\end{equation}
Next we will follow the same line of reasoning of \cite{carlson} to proof the following claim:
\begin{claim}\label{claimA} Let $\epsilon=\epsilon(q):=\min\{1/3,(\log \lambda(q))^{-1}\}$. Then there exists $H_1=H_1(\lambda,v,C)$ such that
\begin{equation}\label{aux00}
\PP\big{(}\max_{z\in R_1}|F(z)|>H_1v(\epsilon^{-1})\lambda(q)\big{)}\leq \frac{1}{2^{q-1}}.
\end{equation}
\end{claim}
\noindent \textit{Proof of the Claim \ref{claimA}.}  By claim \ref{ineq}, for each $q\in\NN$, for all $y\in\RR$
\begin{equation}\label{desigualdade}
\EE|F(\sigma'+iy)|^q \leq D^q \lambda(q)^q v(\epsilon^{-1})^q.
\end{equation}
By condition \textit{i}) above
\begin{equation}\label{auxiliar4}
|F(\sigma'+1+iy)|^q\leq \bigg{(}\sum_{k=1}^\infty \frac{|X_k|}{k^{\frac{3}{2}+\frac{\epsilon}{2}}} \bigg{)}^q\leq C^q\bigg{(} \sum_{k=1}^\infty \frac{\ind_{\EE [|X_k|>0]}}{k^{3/2}}\bigg{)}^q= C^q v (2)^{2q}.
\end{equation}
Let $H=4e\max\{D,C v (2)^{2}\}$ and $A_q$ and $B_q$ the events
$$A_q:=\Omega^*\cap\bigg{[}\int_{-e^q}^{e^q}|F(\sigma'+iy)|^q dy\geq
(H \lambda(q) v(\epsilon^{-1}))^{q}  \bigg{]},$$
$$B_q:=\Omega^*\cap\bigg{[}\int_{\sigma'}^{\sigma'+1}\int_{-e^q}^{e^q}|F(x+iy)|^q dxdy\geq (H \lambda(q) v(\epsilon^{-1}))^{q}        \bigg{]}.$$
By Fubini's Theorem and $(\ref{desigualdade})$:
$$\EE\int_{-e^q}^{e^q} |F(\sigma'+iy)|^q dy=
\int_{-e^q}^{e^q} \EE|F(\sigma'+iy)|^q dy
\leq 2e^qD^q \lambda(q)^q v(\epsilon^{-1})^q,$$
$$\EE\int_{\sigma'}^{\sigma'+1}\int_{-e^q}^{e^q} |F(x+iy)|^q dxdy=\int_{\sigma'}^{\sigma'+1}\bigg{(}\EE\int_{-e^q}^{e^q} |F(x+iy)|^q dy
\bigg{)}dx\leq 2e^qD^q \lambda(q)^q v(\epsilon^{-1})^q.$$
Thus we obtain by Markov's inequality that $\PP(A_q)\leq \frac{1}{2^q}$ and $\PP(B_q)\leq \frac{1}{2^q}$ and hence that $\PP(E_q)\geq 1-\frac{1}{2^{q-1}}$ where $E_q=\Omega^*\cap(A_q \cup B_q)^c$. Let $\omega$ be a fixed element of $E_q$. Then
\begin{equation}\label{auxiliar 1}
\int_{-e^q}^{e^q} |F_\omega(\sigma'+iy)|^q dy<(H \lambda(q) v(\epsilon^{-1}))^{q},
\end{equation}
\begin{equation}\label{auxiliar 2}
\int_{\sigma'}^{\sigma'+1}\int_{-e^q}^{e^q} |F_\omega(x+iy)|^q dxdy<(H \lambda(q) v(\epsilon^{-1}))^{q}.
\end{equation}
\noindent \textit{The choice of $\tau(\omega)$ and $\tau'(\omega)$}. For each $\omega\in E_q$ we can choose $\tau'(\omega)$ and $\tau(\omega)$ in $[e^{q-1},e^q]$ such that
\begin{equation}\label{auxiliar 3}
V_j(\omega,q)\leq (H \lambda(q) v(\epsilon^{-1}))^{q},\;j=1,2,3,4.
\end{equation}
holds for each $\omega\in E_q$. To show the existence of such $\tau$ and $\tau'$, let
$$u(y)=u_\omega(y)=\int_{\sigma'}^{\sigma'+1}|F_\omega(x+iy)|^q dx\quad(\omega\in E_q).$$
Let $L=(H \lambda(q) v(\epsilon^{-1}))^{q}$, $a=e^{q-1}$ and $b=e^q$. By $(\ref{auxiliar 2})$ and Fubini's Theorem:
$$\int_{a}^{b}u(y)dy\leq \int_{-e^q}^{e^q} u(y)dy \leq L.$$
Observe that $b-a>1$. Denote by $m$ the Lebesgue mesure on $\RR$. We claim that:
$$m(\{y\in[a,b]:u(y)\leq L\})>0.$$
Indeed,
$$m(\{y\in[a,b]: u(y)> L\})+m(\{y\in[a,b]: u(y)\leq L\})=b-a,$$
and since $L>0$ and $u\geq 0$ we get:
$$L m(\{y\in[a,b]: u(y)> L\})\leq\int_{a}^b u(y)\ind_{[u> L]}dm(y)\int_{-e^q}^{e^q} u(y)dy \leq L.$$
Hence $m(\{y\in[a,b]: u(y)\geq L\})\leq 1 <b-a$. This shows that the set $\{y\in[a,b]:u(y)\leq L\}$ is not empty and hence the existence of at least one $\tau(\omega)\in[e^{q-1},e^q]$ such that $(\ref{auxiliar 3})$ is satisfied for $j=4$. A similar argument shows the existence of $\tau'\in[e^{q-1},e^q]$ such that $(\ref{auxiliar 3})$ is satisfied for $j=2$. Since $\tau,\tau'\leq e^q$, (\ref{auxiliar4}) and (\ref{auxiliar 1}) gives the desired inequality for $V_1(\omega,q)$ and $V_3(\omega,q)$.

\noindent By condition \textit{iii)} above, $\lambda(q)\leq \lambda(0)e^{\gamma q}$. Since $\epsilon^{-1}(q)=\max\{3,\log \lambda(q)\}$, we obtain that
$$\theta=\sup_{q\geq 1}\epsilon^{-\frac{1}{q}}(q)<\infty.$$ Let $H_1=4\theta H$. By $(\ref{aux60})$ and $(\ref{auxiliar 3})$ we obtain for each $\omega\in E_q$:
$$\max_{z\in R_1}|F_\omega(z)|\leq \bigg{(}\frac{1}{\epsilon(q)} \sum_{j=1}^4 V_j\bigg{)}^{1/q} \leq \bigg{(}\frac{4(H\lambda(q)v(\epsilon^{-1}))^q}{\epsilon(q)}\bigg{)}^{1/q}\leq H_1 \lambda(q)v(\epsilon^{-1}),$$
completing the proof of the Claim \ref{claimA}.

\begin{claim}\label{claimB} Let $\sigma=\frac{1}{2}+\frac{1}{\log \psi (t)}$. Denote $R=R(t):=[\sigma,\frac{4}{3}]\times[-t,t]$.  Then for almost all $\omega\in\Omega$ there exists a real number $t_0=t_0(\omega)$ such that for all $t\geq t_0$:
\begin{equation}\label{ultimato}
\max_{z\in R(t)}|F_\omega(z)|\leq H_2 \psi(t)v(\log \psi(t)),
\end{equation}
where $H_2=H_2(\lambda,v,C)$.
\end{claim}
\noindent \textit{Proof of the Claim \ref{claimB}.} Claim \ref{claimA} implies that
$$\sum_{q=1}^\infty\PP\big{(}\max_{z\in R_1(q)}|F_\omega(z)|\geq H_1 \lambda(q)v(\epsilon^{-1}) \big{)}\leq\sum_{q=0}^\infty \frac{1}{2^q}=2.$$
The Borel-Cantelli Lemma gives a set $\Omega'$ of $\PP(\Omega')=1$ such that for each $\omega\in\Omega'$, there exists $q_0(\omega)\in\NN$, such that for the following inequality holds for all integers $q\geq q_0$:
\begin{equation}\label{aux11}
\max_{z\in R_1(q)}|F_\omega(z)|\leq H_1 \lambda(q)v(\epsilon^{-1}(q)).
\end{equation}
For $x\geq 0$ denote $[x]$ the integer part of $x$. Let $t_0(\omega) = e^{q_0+10}$. For each $t\geq t_0$ let $q(t)=3+[\log t]$. Since
$\log t\leq [\log t]+1\leq q-2$, we get that $t\leq e^{q-2}$ and
$$\log\psi(t)=\log\lambda(\log t)\leq\log\lambda(q-2)\leq \log\lambda(q)=\epsilon^{-1}(q).$$
Hence $R(t)\subset R_1(q)$. By (\ref{aux11}),
$$\max_{z\in R(t)}|F_\omega(z)|\leq \max_{z\in R_1(q(t))}|F_\omega(z)|\leq H_1 \lambda(3+[\log t])v(\log \lambda(3+[\log t])).$$
Observe that $[\log t]\leq \log t$ and $\lambda(3+[\log t])\leq e^{3\gamma}\lambda (\log t)=e^{3\gamma}\psi (t)$. Also,
$v(\log \lambda(3+[\log t]))\leq v(3\gamma +\log \psi(t))$. Let $a(t)=\frac{1}{3\gamma+\log\psi(t)}$ and $b(t)=\frac{1}{\log \psi(t)}$. Then $\lim_{t\to\infty}\frac{b(t)}{a(t)}=1$  and hence
$$b(t)-a(t)\leq 3\gamma a(t)b(t)\ll a^2(t).$$
By Lemma \ref{polo}, $v(3\gamma +\log \psi(t))=v(\log\psi(t))+O(1)$. Hence there exists a constant $D_1=D_1(v)$ such that for all large $t$, $v(\log \lambda(3+[\log t]))\leq D_1v(\log \psi(t) )$. We complete the proof of the claim by choosing $H_2=e^{3\gamma}D_1H_1$.

\noindent \textit{End of the Proof of Lemma \ref{boundrds}.} Let $\sigma(t)=\frac{1}{2}+\frac{1}{\log\psi(t)}$, $1/2<x\leq 1$ and $\Omega'$ be as in clam \ref{claimB}. In the sequel, $\omega\in\Omega'$ is fixed and $t_1=t_1(\omega)$ is a large number such that $(\ref{ultimato})$ holds for all $t\geq t_1$ and $\sigma(t)<x$. Since $|F_\omega(x-it)|=|F_\omega(x+it)|$, it is sufficient to prove Lemma \ref{boundrds} for $t>t_1(\omega)$. Let $\beta=\beta(t)=\log \psi(t)$ and $C_1,C_2,C_3$ be concentric circles with center $\beta+it$ and passing trough the points: $\sigma+\frac{1}{2}+it$, $x+it$ and $\sigma+it$ respectively. Thus, the respective radius of $C_1,C_2,C_3$ are:
\begin{align*}
r_1=&\beta-\sigma-\frac{1}{2},\\
r_2=&\beta-x,\\
r_3=&\beta-\sigma.
\end{align*}
Denote $M_j=M_j(t,\omega)=\max_{z\in C_j}|F_\omega(z)|$, $j=1,2,3$. Since $F_\omega$ is analytic in $\HH_{1/2}$, the Hadamard Three-Circles Theorem states that
\begin{equation}\label{aux1}
M_2\leq M_1^{1-a}M_3^{a},
\end{equation}
where $a=\frac{\log(\frac{r_2}{r_1})}{\log(\frac{r_3}{r_1})}$.
In the sequel we will estimate $M_1$, $M_3$ and $a$ separately.

\noindent \textit{Estimative for $M_1$}. Since for all $k\in\NN$, $|X_k|\leq C$,
\begin{equation}\label{aux2}
M_1\leq \sum_{k=1}^\infty\frac{|X_k|}{k^{1+\beta^{-1}(t)}}\leq C v(\log\psi(t))^2.
\end{equation}

\noindent \textit{Estimative for $M_3$}. By condition \textit{iii}), the function $\lambda$ satisfies $\lambda(c+d)\leq e^{\gamma c}\lambda(d)$ for all $c,d\geq 0$.
In particular, $\psi(t)=\lambda(\log t)\leq \lambda(0)t^\gamma$. Hence $\beta(t)=\log\psi(t)\leq \log (\lambda(0)t^\gamma)=\log\lambda(0)+\gamma\log t$. This gives $\log(t+\beta(t))=\log(1+\beta(t)/t) + \log(t)$ and hence:
$$\psi(t+\beta(t))=\lambda(\log(t+\beta(t)))=\lambda(\;\log(1+\beta(t)/t)\;+\;\log t\;)\leq (1+\beta(t)/t)\lambda(\log t).$$
In particular, we obtain that $\psi(t+\beta(t))\ll \psi(t)$. Also we get that
$$v(\log \psi(t+\beta(t))\leq v(\,\log(1+\beta(t)/t)\;+\;\beta(t) \,).$$
Since $\beta(t)/t=o(1)$, $\log(1+\beta(t)/t)\sim \beta(t)/t$. By Lemma \ref{polo}, we obtain that
$v(\,\log(1+\beta(t)/t)\;+\;\beta(t) \,)=v(\beta(t))+O(1)$. These estimates combined with $(\ref{ultimato})$ gives:
\begin{equation}\label{ultimato2}
M_3\ll \max_{z\in R(t+\beta)} |F_\omega(z)|\leq H_2 \psi(t+\beta(t))v(\log\psi(t+\beta(t)))\ll \psi(t)v(\log\psi(t)).
\end{equation}

\noindent\textit{Estimative for $a(t)$.} We claim that $a(t)=2-2x+O(\beta^{-1}(t))$. Denote $\tau=\beta^{-1}$. Observe
that
$$\frac{r_2}{r_1}=1+\tau \frac{1-x+\tau}{1-\tau(\sigma-1/2)},\quad \frac{r_3}{r_1}=1+ \frac{\tau/2}{1-\tau(\sigma-1/2)}.$$
Using that for $\varphi$ small, $\log(1+\varphi)=\varphi+O(\varphi^2)$ and that $\frac{1}{\tau+O(\tau^2)}=\frac{1}{\tau}+O(1)$, we obtain:
\begin{align*}
a=&\bigg{(}\tau \frac{1-x+\tau}{1-\tau(\sigma-1/2)}+O(\tau^2)\bigg{)}\bigg{(} \frac{2(1-\tau(\sigma-1/2))}{\tau}+O(1) \bigg{)}\\
=&2(1-x)+2\tau+O(\tau)+O(\tau)+O(\tau^2)\\
=&2-2x+O(\tau).
\end{align*}

\noindent \textit{Estimative for $M_1^{1-a}$.} Let $\tau=\frac{1}{\log\psi(t)}$. First observe that $v(\epsilon^{-1})^2 \leq \zeta(1+\epsilon)\sim \epsilon^{-1}.$
Hence
$$v(\log\psi(t))^{O(\tau)}= \exp(\;\;\log (v(\log\psi(t)))\cdot O(\tau)\;)=\exp(\;O(\log \log \psi(t))\cdot O(\tau)\;)=O(1).$$
Recalling $(\ref{aux2})$, we obtain $M_1^{1-a}\ll v(\log\psi(t))^{4x-2}$.

\noindent\textit{Estimative for $M_3^{a}$.} Since $\psi(t)^{O(\tau)}=O(1)$ and $v(\log\psi(t))^{O(\tau)}=O(1)$, we obtain
$$M_3^a\ll\psi(t)^{2-2x}v(\log(\psi(t)))^{2-2x}.$$

\noindent \textit{Estimative for $F_\omega(x+it)$.} Observe that $F_\omega(x+it)\leq M_2(t,\omega)$. Collecting the estimates above, by $(\ref{aux1})$ we get
$$M_2\ll \psi(t)^{2-2x}v(\log\psi(t))^{2x},$$
completing the proof. \end{proof}
\begin{proof}[Proof of Theorem \ref{marcin}] Assume that $\{X_k\}_{k\in\NN}$ satisfies condition \textit{i}) above. If this random variables are independent then by Proposition \ref{kolmogorovlandau} it also satisfies condition \textit{ii}). The condition \textit{iii}) with $\lambda(q)=C\sqrt{q+1}$ ($q\geq 2$) for some constant $C>0$ is the Marcinkiewicz-Zygmund inequality for independent random variables (see \cite{chow} p. 366). Hence $\psi(t)\ll\sqrt{\log t}$.

\noindent If $\{X_k\}_{k\in\NN}$ is a martingale difference that satisfies \textit{i}) above, then, $X_1=M_1-M_0$ and $X_k=M_k-M_{k-1}$ where $(M_n,\mathcal{F}_n)_{n\geq0}$ is a martingale with bounded increments. Hence for any sequence of real numbers $\{a_k\}_{k\in\NN}$, $S_n:=\sum_{k=1}^na_kX_k$ also is a martingale with same filtration $\{\mathcal{F}_n\}_{n\in\NN}$. The condition \textit{iii}) with $\lambda(q)=C_1(q+1)$ ($q\geq 2$) for some $C_1>0$ is the Burkh\"older inequality applied for $S_n$ (see \cite{shiryaev} p. 499). Hence $\psi(t)\ll\log t$. Let $S_n(\epsilon):=\sum_{k=1}^n\frac{X_k}{k^{\frac{1+\epsilon}{2}}}$. For $q=2$,  the Burkh\"older inequality applied for $S_n(\epsilon)$ gives that $\EE|S_n(\epsilon)|^2\leq D\lambda(2) \zeta(1+\epsilon)$ and hence that $\sup_{n\in\NN} \EE |S_n(\epsilon)|^2<\infty$. By Doob's martingale convergence Theorem (see \cite{shiryaev} p. 510) we obtain the almost sure convergence of $S_n(\epsilon)$ and hence the almost sure convergence of the Dirichlet series $\sum_{k=1}^\infty \frac{X_k}{k^ {\frac{1+\epsilon}{2}} }$ for each $\epsilon>0$. The referred properties for the convergence of Dirichlet series stated in the proof of Proposition \ref{kolmogorovlandau} gives that $\{X_k\}_{k\in\NN}$ satisfy Theorem \ref{boundrds} condition \textit{ii}).

\noindent Given a probability space $(\Omega,\mathcal{F},\PP)$ let $\mathcal{F}_1$ and $\mathcal{F}_2$ be sub-sigma algebras of $\mathcal{F}$. For $j=1,2$, denote $\\ L^2(\mathcal{F}_j)=\{X:\Omega\to\RR:\EE |X|^2<\infty\mbox{ and } X\mbox{ is } \mathcal{F}_j-\mbox{ measurable}\}$ and $\|X\|_2=\sqrt{\EE X^2}$. Let
$$\rho(\mathcal{F}_1,\mathcal{F}_2)=\sup\{\VV f_1f_2/(\|f_1\|_2\|f_2\|_2):f_1\in L^2(\mathcal{F}_1)\mbox{ and }f_2\in L^2(\mathcal{F}_2)\}.$$
Let $\{X_k\}_{k\in\NN}$ be a sequence of random variables and for $S\subset\NN$, let $\mathcal{F}_S$ be the sigma algebra generated by the random variables $\{X_k\}_{k\in S}$. Define
$$\rho^*(n)=\sup\{\rho(\mathcal{F}_S,\mathcal{F}_T):S,T\subset\NN\mbox{ and }\min_{s\in S,t\in T}|s-t|\geq n\}.$$
One says that the sequence $\{X_k\}_{k\in\NN}$ is $\rho^*$-mixing if $\lim_{n\to\infty} \rho^*(n)=0$ (see \cite{bradleystrong} p. 114). In particular, if
$\{X_k\}_{k\in\NN}$ is $\rho^*$-mixing, then there exists $n\in\NN$ such that $\rho^*(n)<1$. In \cite{bryc} it has been proved a result which implies the following: If $\{X_k\}_{k\in\NN}$ are centered and uniformly bounded random variables with $\rho^*(n)<1$ for some large $n$, then condition \textit{ii}) is satisfied. Moreover, in \cite{bryc} (Lemma 1 and 2 and Remark 4), the condition \textit{iii}) for $q\geq 2$ is satisfied with $\lambda(q)=\frac{1}{1-\rho^{\frac{2}{q}}}\sqrt{q+1}\sim C_2(q+1)^{3/2}$, where $\rho=\rho^*(n)$. Hence $\psi(t)\ll (\log t )^{3/2}$.

\noindent Let $\zeta$ be the Riemann zeta function. We recall that $\zeta(1+\epsilon)\sim \epsilon^{-1}$. Hence $v(\epsilon^{-1})^2 \ll \epsilon^{-1}$. On the other hand, if $X_k=0$ for all non prime $k$, then $v(\epsilon^{-1})^2=\log(\epsilon^{-1})$. Hence for large $t$ $v(\log\psi(t))^2\ll \log\log t$, and if $X_k=0$ for all non prime $k$, $v(\log\psi(t))^2\ll \log\log\log t$.
\end{proof}

\section{Proofs of the main results}\label{provas}
\subsection{(Theorem \ref{alphaa})} Let $R_1\subset R_2$ be open connected sets of $\CC$. An analytic function $h:R_1\to\CC$ has analytic extension to $R_2$ if there exists an analytic function $\bar{h}:R_2\to\CC$ such that $\bar{h}(z)=h(z)$ for all $z\in R_1$. We say that a random analytic function $h:R_1\times\Omega\to\CC$ has analytic extension to $R_2$ if the set of elements $\omega\in\Omega$ for which $h_\omega$ has analytic extension to $R_2$ contains a set $\Omega^*\in\mathcal{F}$ such that $\PP(\Omega^*)=1$.

\begin{proof}[Proof of Theorem \ref{alphaa}] Let $\alpha\in(0,1/2)$ and $f_\alpha$ be the random multiplicative function such that $\EE(f(p))=-\frac{1}{p^\alpha}$ for each prime $p$. Denote $F_\alpha(z):=\sum_{k=1}^\infty\frac{f_\alpha(k)}{k^z}$.\\
\begin{claim}\label{equiv1} The half plane $\HH_{1/2+\alpha}$ is a zero free region for $\zeta$ if and only if $F_\alpha$ has analytic extension to $\HH_{1/2}$.
\end{claim}
\noindent \textit{Proof of the claim}. Since $\{f(p)\}_{p\in\mathcal{P}}$ is a sequence of independent random variables, $\EE f(k)=\frac{\mu(k)}{k^{\alpha}}$. For $z\in\HH_1$ we obtain that $\EE F_\alpha (z)=\sum_{k=1}^\infty \frac{\mu(k)}{k^{z+\alpha}}=\frac{1}{\zeta(z+\alpha)}$. By Lemma \ref{lema0} ii) there exists a random analytic function $\theta:\HH_{1/2}\times\Omega\to\CC$ such that $F_\alpha(z)=\theta(z)\frac{1}{\zeta(z+\alpha)}$. By  iii) of this Lemma,  $\frac{1}{\theta}:\HH_{1/2}\times\Omega\to\CC$ also is a random analytic function.  Hence $\frac{1}{\zeta(z+\alpha)}$ is analytic in $\HH_{1/2}$ if and only if $F_\alpha$ has analytic extension to $\HH_{1/2}$. Since $\zeta$ is analytic in $\CC\setminus\{1\}$ with a simple pole in $z=1$, we obtain that $\zeta$ has no zeros in the half plane $\HH_{1/2+\alpha}$ if and only if $\frac{1}{\zeta(z+\alpha)}$ is analytic in $\HH_{1/2}$, completing the proof of the claim.

\noindent Assume $M_{f_\alpha}(x)=o(x^{1/2+\epsilon})$ for all $\epsilon>0$ \textit{a.s.} By partial summation (Lemma \ref{abell}), the series $F_\alpha:\HH_{1/2}\times\Omega\to\CC$ is a random analytic function. By claim \ref{equiv1} we conclude that $\zeta$ has no zeros in the half plane $\HH_{1/2+\alpha}$. Thus if for each $\alpha>0$, $M_{f_\alpha}(x)=o(x^{1/2+\epsilon})$ $\forall\epsilon>0$ \textit{a.s.}, then $\zeta$ has no zeros in $\HH_{1/2}$.

Assume RH. In \cite{little} J.E.Littlewwod proved, for fixed $\sigma>1/2$, that RH implies that $\frac{1}{\zeta(\sigma+it)}=o(t^{\delta})$ for all $\delta>0$. By Theorem \ref{marcin}, for fixed $1/2<\sigma\leq1$ we have
$$\theta(\sigma+it)\ll \exp(\log(t)^{1-\sigma}\log\log\log t)=o(t^\delta),\;\forall \delta>0,\;a.s. $$
By claim \ref{equiv1} $F_\alpha$ has analytic extension to $\HH_{1/2}$ given by $\frac{\theta(z)}{\zeta(z+\alpha)}$. Hence $F_\alpha(\sigma+it)\ll t^{\delta}$ for all $\delta>0$ \textit{a.s.} We recall the following result from the theory of the Dirichlet series:  Assume that $G(z)=\sum_{n=1}^\infty\frac{g(n)}{n^z}$ converges absolutely $\forall z\in\HH_1$, and that for some $c<1$, $G$ has analytic extension to $\HH_c$ given by $\bar{G}$. If $\bar{G}(\sigma+it)=o(t^\delta)$ for all $\delta>0$, then $M_g(x)=o(x^{c+\epsilon})$ for all $\epsilon>0$ (see \cite{tenenbaumlivro} page 134, Theorem 4). This result applied for $F_\alpha$ completes the proof. \end{proof}
\subsection{(Theorems  \ref{viz1/2} and \ref{viz1/22})}
\begin{proof}[Proof of Theorem \ref{viz1/2} ] Let $\EE f(p)=-\delta_p$ where $0<\delta_p\leq1$. Since $|f(p)|\leq 1$ $\forall p\in\mathcal{P}$, by the Kolmogorov two series Theorem, $f$ weakly biased implies that $\sum_{p\in\mathcal{P}}\frac{\delta_p}{p}$ converges. On the other hand, for $\alpha\in(0,1/2)$, the convergence of $\sum_{p\in\mathcal{P}}\frac{\delta_p}{p^{1-\alpha}}$ implies that $\sum_{p\in\mathcal{P}}\frac{f(p)}{p^{1-\alpha}}$ converges \textit{a.s.} Thus, we only need to prove that, if $f$ is weakly biased and
\begin{equation}\label{star}
\PP(M_f(x)=o(x^{1-\epsilon})\mbox{ for some }\epsilon>0)=1,
\end{equation}
then there exists $\alpha>0$ such that the series $\sum_{p\in\mathcal{P}}\frac{\delta_p}{p^{1-\alpha}}$ converges.

\noindent Let $F:\HH_1\times\Omega\to\CC$, $v:\HH_{1/2}\times\Omega\to\CC$ and $u:\HH_1\to\CC$ be given by:
\begin{align*}
&F(z)=\sum_{n=1}^\infty\frac{f(n)}{n^z},\\
&v(z)=\sum_{p\in\mathcal{P}}\frac{f(p)-\EE f(p)}{p^z},\\
&u(z)=\sum_{p\in\mathcal{P}}\frac{\delta_p}{p^z}.
\end{align*}
By Proposition \ref{kolmogorovlandau}, $v$ is a random analytic function and $u$ is analytic. By Lemma \ref{lema0} and claim \ref{cl1} there exists a random analytic function $w:\HH_{1/2}\times\Omega\to\CC$ such that for each $z\in\HH_1$
\begin{equation}\label{bran}
F(z)=\exp(v(z)+w(z)-u(z)).
\end{equation}
Since the series $\sum_{p\in\mathcal{P}}\frac{\delta_p}{p}$ converges, we obtain that $\lim_{x\to1^+}u(x)=\sum_{p\in\mathcal{P}}\frac{\delta_p}{p}<\infty$. This combined with $(\ref{bran})$ implies that $\lim_{x\to1+}F(x)>0$, \textit{a.s.} By $(\ref{star})$ there is a set $\Omega^*$ with $\PP(\Omega^*)=1$ such that for each $\omega\in\Omega^*$ there exists $\epsilon=\epsilon(\omega)>0$ for which $M_{f_\omega}(x)=o(x^{1-\epsilon})$. Hence, if $\omega$ is a fixed element of $\Omega^*$, Lemma \ref{abell} implies that the series $F_\omega(z)=\sum_{n=1}^\infty\frac{f_\omega(n)}{n^z}$ converges for each $z\in\HH_{1-\epsilon}$ and it is an analytic function in this half plane. Thus for $\PP-$almost all $\omega\in\Omega$ we obtain an $\epsilon=\epsilon(\omega)>0$ such that $F_\omega(z)$ is analytic in $\HH_{1-\epsilon}$ and satisfies $F_{\omega}(1)\neq 0$. In particular for each of these $\omega$, there exists an open ball $B=B(\omega)\subset \HH_{1-\epsilon}$ with positive radius and centered at $z=1$ such that $F_{\omega}(z)\neq 0$ for all $z\in\HH_{1}\cup B$. Since this random subset of $\CC$ is a simply connected region, $F_\omega$ has a branch of the logarithm $r_\omega:\HH_{1}\cup B\to\CC$ (see \cite{conway}, p. 94-95, Corollary 6.17), i.e., $r_\omega$ is analytic and satisfies $F_\omega(z)=\exp(r_{\omega}(z))$ for all $z\in\HH_1\cup B$. This combined with
$(\ref{bran})$ gives for $\PP-$almost all $\omega$ and all $z\in\HH_1$
\begin{equation}\label{bran2}
\Lambda(z):=\exp(u(z))=\exp(v_\omega(z)+w_\omega(z)-r_\omega(z)).
\end{equation}
In particular $\lambda_\omega(z):=v_\omega(z)+w_\omega(z)-r_\omega(z)$ is analytic in $\HH_1\cup B$ and hence it is, \textit{a.s.}, a branch of the logarithm for the analytic function $\Lambda:\HH_1\to\CC$. A classical result from complex analysis states that there exists an integer $k=k(\omega)$ such that for all $z\in\HH_1$ and almost all $\omega$, $u(z)-\lambda_\omega(z)=2k\pi i$. That is, $\bar{u}_\omega:{\HH_1}\cup B\to\CC$ given by $\bar{u}_\omega(z)=\lambda_\omega(z)+2k\pi i$ extends $u$ analytically to $\HH_1\cup B$. Since for $z\in\HH_1$, $\bar{u}_\omega(z)=u(z)=\sum_{p\in\mathcal{P}}\frac{\delta_p}{p^z}$ is a Dirichlet series of non-negative terms that it is analytic in an open disk centered at $z=1$, a classical result concerning Dirichlet series of this type (see \cite{apostol} p. 237, Theorem 11.13) implies that there is $\alpha>0$ for which the series $\sum_{p\in\mathcal{P}}\frac{\delta_p}{p^{1-\alpha}}$ converges.  \end{proof}
\begin{proof}[Proof of Theorem \ref{viz1/22}] We begin the proof with the following claim:
\begin{claim}\label{claim8} Let $0<\alpha<1/2$. Assume that $\EE X_p=-\frac{\delta_p}{p^\alpha}$ where $0\leq \delta_p\leq 1$, $\limsup\delta_p=\delta<1$, and $\sum_{p\in\mathcal{P}}\frac{\delta_p}{p}=\infty$. Then $M_f(x)$ is not $o(x^{1-\alpha-\epsilon})$ for any $\epsilon>0$, \textit{a.s.}
\end{claim}
\begin{proof}[Proof of the claim.] Let $0< \alpha<1/2$. By Lemma \ref{lema0}, there is a random analytic function $\theta:\HH_{1/2}\times\Omega\to\CC$ such that for all $z\in\HH_1$ and all $\omega\in\Omega$
$$F(z)=\theta(z)\EE F(z),$$
where $F(z)=\sum_{k=1}^\infty\frac{f(k)}{k^z}$. Moreover, since $\EE f(p)=-\delta_p/p^\alpha$, claim \ref{cl1} gives
\begin{equation}\label{ptheo1.31}
\EE F(z)=\exp\bigg{(}-\sum_{p\in\mathcal{P}}\frac{\delta_p}{p^{z+\alpha}}\bigg{)}\exp(A(z))\quad (z\in\HH_1),
\end{equation}
where $A:\HH_{1/2}\to\CC$ is analytic. Since the series $\sum_{p\in\mathcal{P}}\frac{\delta_p}{p^{z+\alpha}}$ converges absolutely for $z\in\HH_{1-\alpha}$ we obtain that the function in the right side of $(\ref{ptheo1.31})$ is analytic in $\HH_{1-\alpha}$ and hence extends analytically $\EE F(z)$ to this half plane. Let $z\in\HH_1$ and $\zeta(z)=\sum_{k=1}^\infty\frac{1}{k^z}$ be the Riemann zeta function. A direct application of claim \ref{cl1} gives that
\begin{equation}\label{ptheo1.32}
\zeta(z):=\exp\bigg{(}\sum_{p\in\mathcal{P}}\frac{1}{p^z}\bigg{)}\exp(B(z)),
\end{equation}
where $B:\HH_{1/2}\to\CC$ is analytic. By combining $(\ref{ptheo1.31})$ and $(\ref{ptheo1.32})$:
\begin{equation}\label{ptheo1.33}
\zeta(1+\epsilon)\EE F(1-\alpha+\epsilon)=\exp\bigg{(}\sum_{p\in\mathcal{P}}\frac{1-\delta_p}{p^{1+\epsilon}}\bigg{)}\exp(A(1-\alpha+\epsilon)+B(1+\epsilon)).
\end{equation}
Since $\limsup \delta_p=\delta<1$ there exists $\eta>0$ such that $1-\delta_p\geq \eta$ for all $p$ sufficiently large. Hence
\begin{equation*}
\sum_{p\in\mathcal{P}}\frac{1-\delta_p}{p}=\infty.
\end{equation*}
This combined with $(\ref{ptheo1.33})$ implies
\begin{equation}\label{ptheo1.34}
\lim_{\epsilon\to 0^+}\zeta(1+\epsilon)\EE F(1-\alpha+\epsilon)=\infty.
\end{equation}
 On the other hand hypothesis $\sum_{p\in\mathcal{P}}\frac{\delta_p}{p}=\infty$ combined with $(\ref{ptheo1.31})$ gives that
\begin{equation}\label{ptheo1.35}
\lim_{\epsilon\to0}\EE F(1-\alpha+\epsilon)=0.
\end{equation}
 Hence if we assume
$$\PP(M_f(x)=o(x^{1-\alpha-\epsilon}\mbox{ for some }\epsilon>0)=1,$$
by Lemma \ref{abell} we obtain for almost all $\omega\in\Omega$ an $\epsilon=\epsilon(\omega)>0$ such that $F_\omega(z)=\sum_{k=1}^\infty \frac{f_\omega(k)}{k^z}$ is analytic in $\HH_{1-\alpha-\epsilon}$. Since $\EE F(z)=\frac{F_\omega(z)}{\theta_\omega(z)}$,
$\EE F(z)$ is analytic in a open neighborhood of $z=1-\alpha$. By $(\ref{ptheo1.35})$, $\EE F(1-\alpha)=0$ while $(\ref{ptheo1.34})$
gives that this can not be an zero of an analytic function, since the Riemann zeta function has a simple pole at $z=1$. This gives a contradiction which implies that $\EE F(z)$ is not analytic in $z=1-\alpha$, and hence that
$$\PP(M_f(x)=o(x^{1-\alpha-\epsilon}\mbox{ for some }\epsilon>0)<1.$$
A direct application of Corollary \ref{caldal} implies that this probability is zero.  \end{proof}
\begin{remark}[Uniform coupling]\label{uniform coupling}
Let $(\Omega,\mathcal{F},\PP)$ be the probability space where $\Omega$ is the set of the sequences $\omega=(\omega_p)_{p\in\mathcal{P}}$ such that $\omega_p\in[0,1]$ for each prime $p$, $\mathcal{F}$ is the Borel sigma-algebra of $\Omega$ and $\PP$ is the Lebesgue product measure in $\mathcal{F}$. For a random multiplicative function $f:\NN\times\Omega\to\{-1,0,1\}$, let $f(p):\Omega\to\{-1,1\}$ be the random variable given by
\begin{equation}\label{indicator}
f_\omega(p)=\,\ind_{(a_p,1]}(\omega_p)-\ind_{[0,a_p]}(\omega_p),
\end{equation}
where $a_p:=\PP(f(p)=-1)$. If $f(p)$ and $g(p)$ are random variables given by $(\ref{indicator})$ then
$$|\PP(f(p)\neq g(p))|=|\PP(f(p)=-1)-\PP(g(p)=-1)|.$$
\end{remark}
\begin{definition}\label{uniform def} Let $f$ and $g$ be random multiplicative functions. We say that $f$ and $g$ are uniformly coupled if they are defined in $(\Omega,\mathcal{F},\PP)$ as in Remark \ref{uniform coupling} and $\forall p\in\mathcal{P}$, $f(p)$ and $g(p)$ are given by $(\ref{indicator})$.
\end{definition}

Let $\alpha>0$ and assume that $\{\delta_p\}_{p\in\mathcal{P}}$ is such that $0\leq \delta_p\leq 1$ $\forall p\in\mathcal{P}$. Let $f$ be a random multiplicative function such that for each prime $p$, $\{f(p)\}_{p\in\mathcal{P}}$ is given by (\ref{indicator}) with $a_p=\frac{1}{2}+\frac{\delta_p}{2p^\alpha}$ and hence $\EE f(p)=-\frac{\delta_p}{p^\alpha}$.
Let $u,h:\NN\times\Omega\to\{-1,0,1\}$ be random functions that satisfy the multiplicative property (\ref{multiplicative}), and such that for each prime $p$, $u(p)=Z_p$ and $h(p)=W_p$ where,
\begin{align*}
&Z_p(\omega):=-\ind_{[0,\frac{1}{2}-\frac{\delta_p}{2p^\alpha})}(\omega_p)+\ind_{(\frac{1}{2} + \frac{\delta_p}{2p^\alpha},1]}(\omega_p),\\
&W_p(\omega):=-\ind_{[\frac{1}{2}-\frac{\delta_p}{2p^\alpha},\frac{1}{2}+\frac{\delta_p}{2p^\alpha}]}(\omega_p).
\end{align*}
\begin{claim}\label{series} Let $\gamma=\max\{1/2,1-\alpha\}.$ Then $\sum_{n=1}^\infty\frac{u(n)}{n^{1/2+\epsilon}}$ and $\sum_{n=1}^\infty\frac{|h(n)|}{n^{\gamma+\epsilon}}$ converges $\forall\epsilon>0$ \textit{a.s.}
\end{claim}
\begin{proof}[Proof of the claim] The Rademacher-Menshov Theorem \cite{meaney} states that if $\{X_n\}_{n\in\NN}$ is a sequence of orthogonal random variables such that
the series $\sum_{n=1}^\infty \log^2(n+1) \EE X_n^2 $ converges and $\EE X_n=0$ for all $n$, then the random series $\sum_{n=1}^\infty X_n$ converges \textit{a.s.} If $k$ and $l$ are distinct squarefree integers, there are at least one prime $p$ such that either $p|k$ or $p|l$ while $p$ do not divide $\gcd(k,l)$, and hence $\EE u(k)u(l)=0$. Since $|u(n)|\leq 1$ $\forall n\in\NN$, by the Rademacher-Menshov Theorem, $\sum_{n=1}^\infty\frac{u(n)}{n^{1/2+\epsilon}}$ converges $\forall \epsilon>0$ \textit{a.s.}

\noindent By the Kolmogorov two series Theorem, $\sum_{p\in\mathcal{P}}\frac{|h(p)|}{p^{\gamma+\epsilon}}$ converges $\forall \epsilon>0$ \textit{a.s.} since
\begin{align*}
\sum_{p\in\mathcal{P}}\frac{\EE |h(p)|}{p^{\gamma+\epsilon}}\leq& \sum_{p\in\mathcal{P}}\frac{1}{p^{\gamma+\alpha+\epsilon}}<\infty,\\
\sum_{p\in\mathcal{P}}\frac{\VV |h(p)|}{p^{2\gamma+2\epsilon}}\leq&\sum_{p\in\mathcal{P}}\frac{1}{p^{1+2\epsilon}}<\infty.
\end{align*}
We recall a classical result for a Dirichlet series of an multiplicative function $\phi:\NN\to[-1,1]$ which states that: If for each $p\in\mathcal{P}$, $\phi(p^m)=0$ for all $m\geq2$, then for each $\sigma>0$ the series $\sum_{n=1}^\infty\frac{|\phi(n)|}{n^\sigma}$ converges if and only if the series $\sum_{p\in\mathcal{P}}\frac{|\phi(p)|}{p^{\sigma}}$ converges (see \cite{tenenbaumlivro}, p. 106 Theorem 2). A direct application of this result for $h$ completes the proof of the claim. \end{proof}

\noindent The Dirichlet convolution between $u$ and $h$, denoted by $u\ast h$ is given by: $(u\ast h)(n):=\sum_{d|n}u(d)h(n/d)$. Since for each prime $p$, $f(p)=u(p)+h(p)$ and $u(p) \cdot h(p)=0$, we obtain that
\begin{align*}
u\ast h(p)=&u(p)+h(p)=f(p),\\
u\ast h(p^n)=&\sum_{k=0}^nu(p^k)h(p^{n-k})=0\quad(n\geq 2).
\end{align*}
This implies that for each prime $p$, $f(p^m)=u\ast h(p^m)$ $\forall m\in\NN$. Since the convolution between two multiplicative functions results in a multiplicative function, we conclude that $f=u\ast h$. A result for Dirichlet series states that if $\sum_{k=1}^\infty \frac{h(k)}{k^\sigma}$ converges absolutely and if $\sum_{k=1}^\infty \frac{u(k)}{k^\sigma}$ converges then $\sum_{k=1}^\infty \frac{(u\ast h)(k)}{k^\sigma}$ also converges (see \cite{tenenbaumlivro}, p. 122, Notes 1.1). This combined with claim \ref{series} implies that $\sum_{k=1}^\infty\frac{f(k)}{k^{\gamma+\epsilon}}$ converges \textit{a.s.} A direct application of Kronecker's Lemma gives that $M_f(x)=o(x^{\gamma+\epsilon})$ for all $\epsilon>0$ \textit{a.s.}, completing the proof of Theorem \ref{viz1/22} \end{proof}
\subsection{(Theorems \ref{alphaa2} and \ref{alphaa1})}
\begin{remark}\label{interinter} Let $0=a_1<a_2<...<a_{n+1}=1$ and $I_k=[a_k,a_{k+1})$ if $1\leq k<n$ and $I_n=[a_n,a_{n+1}]$. Let $\psi:[0,1]\to[0,1]$ be a bijection such that for all $1\leq k\leq n$, $\psi:I_k\to [0,1]$ is a translation. Then $\psi$ is called interval exchange transformation (see \cite{viana}). If $m$ denotes the Lebesgue measure on $[0,1]$, then for each Borelian $B\subset[0,1]$, $m(\psi^{-1}(B))=m(B)$, i.e., an interval exchange preserves the Lebesgue measure. If $(\Omega,\mathcal{F},\PP)$ is the probability space introduced in the remark \ref{uniform coupling} and if  $\psi_p:[0,1]\to[0,1]$ is an interval exchange transformation for all $p\in\mathcal{P}$ then $T:\Omega\to\Omega$ given by
$$T (\omega_2, \omega_3, \omega_5,...)=(\psi_2(\omega_2), \psi_3(\omega_3), \psi_5(\omega_5),...)$$
preserves $\PP$, i.e., for each $B\in\mathcal{F}$, $\PP(T^{-1}(B))=\PP(B)$.
\end{remark}
We say that a random multiplicative function $g$ supported on the squarefree integers is biased towards $|\mu|$ if $\EE g(p)>0$ $\forall p\in\mathcal{P}$. In the sequel, we will use the advantage of the probability space (uniform coupling) introduced in the remark \ref{uniform coupling} where it is defined the measure preserving transformation $T$ introduced in the remark \ref{interinter}. This will enable us to transport some properties of a biased $g$ towards $|\mu|$ to a strongly biased random multiplicative function towards $\mu$.
\begin{lemma}\label{left} Let $g$ be a random multiplicative function biased to $|\mu|$. Let $z\in\HH_1$ and $G(z):=\sum_{k=1}^\infty \frac{g(k)}{k^z}$. If for some $0<\alpha<1/2$ there exists a random analytic function $\bar{G}:\HH_{1-\alpha}\times\Omega\to\CC$ such that for all $z\in\HH_1$, $\bar{G}(z)=G(z)$, then $\sum_{p\in\mathcal{P}}\frac{g(p)}{p^{1-\alpha+\epsilon}}$ converges $\forall\epsilon>0$ \textit{a.s.}
\end{lemma}
\begin{proof} Since the random variables $\{g(p)\}_{p\in\mathcal{P}}$ are independent, $k\in\NN\mapsto \EE g(k)$ is multiplicative, supported on the square free integers and non-negative. In particular, $\EE G(z)$ is a Dirichlet series of non-negative terms. By Lemma \ref{lema0}, for all $z\in\HH_1$ there exists a non-vanishing random analytic function $\theta:\HH_{1/2}\times\Omega \to\CC$ such that for all $z\in\HH_1$, $G(z)=\EE G(z)\theta(z)$. In particular $\frac{1}{\theta}$ also is a random analytic function. Hence $\Lambda:\HH_{1-\alpha}\times\Omega\to\CC$ given by
$$\Lambda(z)=\frac{\bar{G}(z)}{\theta(z)}$$
is a random analytic function and satisfies $\Lambda(z)=\EE G(z)$ for all $z\in\HH_1$. In particular there exists $\omega\in\Omega$ such that
$\Lambda_\omega:\HH_{1-\alpha}\to\CC$ is analytic and  $\Lambda_\omega(z)=\EE G(z)$ for all $z\in\HH_1$. We recall that if a Dirichlet series of non-negative terms has analytic extension to the half plane $\HH_{1-\alpha}$, then actually this series converges for all $z$ in this half plane. Hence
$\sum_{k=1}^\infty \frac{\EE g(k)}{k^z}$ converges for every $z\in\HH_{1-\alpha}$. Since $k\in\NN\mapsto\EE g(k)$ is multiplicative and non-negative,
the series $\sum_{p\in\mathcal{P}}\frac{\EE g(p)}{p^z}$ converges for all $z\in\HH_{1-\alpha}$ (see \cite{tenenbaumlivro} p. 106, Theorem 2 and remark (a)). By Proposition \ref{kolmogorovlandau} we obtain that $\sum_{p\in\mathcal{P}}\frac{g(p)}{p^{1-\alpha+\epsilon}}$ converges $\forall\epsilon>0$ \textit{a.s.} \end{proof}

\begin{proof}[Proof of Theorem \ref{alphaa1}]  Let $u$ and $g$ be random multiplicatives function such that $u(p)$ and $g(p)$ are given by $(\ref{indicator})$ with
\begin{align*}
&\PP(u(p)=-1)=1/2,\\
&\PP(g(p)=-1)=1/2-\PP(f(p)=1),
\end{align*}
respectively. Hence $u$ is unbiased and
$$\EE g(p) = 2\PP(f(p)=1)= 1+\EE f(p)>0.$$
\noindent Denote $F(z):=\sum_{k=1}^\infty \frac{f(k)}{k^z}$,
$U(z)=\sum_{k=1}^\infty \frac{u(k)}{k^z}$ and $G(z)=\sum_{k=1}^\infty \frac{g(k)}{k^z}$. Let $\varphi:\HH_{1}\times\Omega\to\CC$ and $\psi:\HH_1\times\Omega\to\CC$ be the random analytic functions
$$\varphi(z)=\zeta(z)F(z)\mbox{ and }\psi(z)=U^{-1}(z)G(z),$$
where $U^{-1}(z)=\frac{1}{U(z)}$. For a random analytic function $\lambda:\HH_1\times\Omega\to\CC$ denote
$$A_\lambda:=\{\omega\in\Omega: \lambda_\omega \mbox{ has analytic extension to }\HH_{1-\alpha}\}.$$
By Proposition $(\ref{mensu})$, if $\lambda_\omega:\HH_1\to\CC$ is analytic for all $\omega\in\Omega$, then $A_\lambda$ is measurable. In particular the events $A_F$, $A_G$, $A_U$, $A_{U^{-1}}$, $A_\varphi$, and $ A_\psi$ are measurable.
\begin{claim}\label{inter} $\PP(A_\psi)=\PP(A_G)$ and under the hypothesis of Theorem \ref{alphaa1}, $\PP(A_\varphi)=1$.
\end{claim}
\noindent \textit{Proof of the claim.} Observe that $\EE u(k)=0$ if $k>1$ and $\EE u(1)=1$, hence $\EE U(z)=1$. By Lemma \ref{lema0}, $\PP(A_U,A_{U^{-1}})=1$. Since $\psi(z)=U^{-1}(z)G(z)$ and $G(z)=\psi(z)U(z)$:
\begin{align*}
\PP(A_G)=&\PP(A_G \cap A_{U^{-1}})\leq\PP(A_\psi),\\
\PP(A_\psi)=&\PP(A_\psi\cap A_{U})\quad\leq\PP(A_G).
\end{align*}
Hence $\PP(A_G)=\PP(A_\psi)$, completing the first statement of the claim. By hypothesis we have that $M_{f}(x)=o(x^{1-\alpha})$ \textit{a.s.} A direct application of Lemma \ref{abell} implies that $\PP(A_F)=1$. In addition, hypothesis $\EE f(p)<0$ and $\sum_{p\in\mathcal{P}}\frac{\EE f(p)}{p}=-\infty$ implies that $\lim_{\epsilon\to 0^+} \EE F(1+\epsilon)=0$. By applying Lemma \ref{lema0} iii) to $F$ we obtain that $\PP(A_F\cap[F(1)=0])=1$. If $\omega\in A_F\cap[F(1)=0]$, then $F_\omega:\HH_{1-\alpha}\to\CC$ is analytic and $F_\omega(1)=0$, hence there exists an integer $m=m(\omega)\geq1$ such that $\frac{F_{\omega}(1)}{(z-1)^m}$ is analytic in $\HH_{1-\alpha}$ (see \cite{conway} p. 79, Corollary 3.9). Thus $A_F\cap[F(1)=0]\subset A_\varphi$ since, the Riemann zeta function extends analytically to $\CC\setminus\{1\}$ with a simple pole at $z=1$, and hence this simple pole cancel \textit{a.s.} with the zero at $z=1$ of the random analytic function $F$. Hence, $\PP(A_\varphi)=1$, completing the proof of the claim.

Let $I_p$ and $J_p$ be the intervals
\begin{align*}
I_p:=&\bigg{(}\frac{1}{2}-\PP(f(p)=1),\frac{1}{2}\bigg{]},\\
J_p:=&\bigg{(}1-\PP(f(p)=1),\;\;1\bigg{]}.
\end{align*}
Observe that $I_p$ and $J_p$ have the same Lebesgue measure. For each $p\in\mathcal{P}$ let $\psi_p:[0,1]\to[0,1]$ be the interval exchange transformation that exchanges only $J_p$ and $I_p$ such that $\psi_p(I_p)=J_p$ and $\psi(J_p)=I_p$. Let $T:\Omega\to\Omega$ be the measure preserving transformation as in the remark \ref{interinter} and $\omega^*=T(\omega)$. We claim that for each $\omega\in\Omega$, $\varphi_\omega=\psi_{\omega^*}$. Indeed the Euler product representation (see \cite{tenenbaumlivro} p. 106) for $F$, $U$, $G$ and $\zeta$ allow us to deduce the functional equations which holds for all $z\in\HH_1$:
\begin{align*}
\varphi_{\omega}(z)&=\prod_{p\in\mathcal{P}}\frac{p^z+\ind_{J_p}(\omega_p)}{p^z-\ind_{J_p}(\omega_p)},\\
\psi_\omega(z)&=\prod_{p\in\mathcal{P}}\frac{p^z+\ind_{I_p}(\omega_p)}
{p^z-\ind_{I_p}(\omega_p)}.
\end{align*}
Since $\ind_{J_p}(\psi_p(\omega_p))=\ind_{I_p}(\omega_p)$ we obtain that $\varphi_\omega(z)=\psi_{\omega^*}(z)$ for all $z\in\HH_1$ and hence that $T^{-1}(A_\psi)=A_{\varphi}$. This combined with claim \ref{inter} implies that
$$\PP(A_G)=\PP(A_\psi)=\PP(T^{-1}(A_\psi))=\PP(A_\varphi)=1.$$
By Lemma \ref{left} we then conclude that $\sum_{p\in\mathcal{P}}\frac{g(p)}{p^{1-\alpha+\epsilon}}$ converges $\forall\epsilon>0$. Since $|g(p)|\leq 1$, By Kolmogorov two series Theorem, $\sum_{p\in\mathcal{P}}\frac{\EE g(p)}{p^{1-\alpha+\epsilon}}$ converges $\forall\epsilon>0$. Since $\EE g(p)=1+\EE f(p)$ we then conclude, by Proposition \ref{kolmogorovlandau} the proof of Theorem \ref{alphaa1}. \end{proof}
\begin{proof}[Proof of Theorem \ref{alphaa2}] Let $\varphi$, $\psi$, $G$, and $U$ be as in the proof Theorem \ref{alphaa1}. Recall from Lemma \ref{lema0} and Theorem \ref{marcin} that, for fixed $\sigma\in(1/2,1]$
\begin{equation}\label{antepenultima}
U^{-1}(\sigma+it)\ll \exp((\log t)^{1-\sigma}\log\log\log t )\;a.s.
\end{equation}
Moreover, Lemma \ref{lema0} applied for $G$ gives that $G(z)=\EE G(z)\theta(z)$ ($z\in\HH_1$),
where $\theta:\HH_{1/2}\times\Omega\to\CC$ is a non-vanishing random analytic function that satisfies for fixed $\sigma\in(1/2,1]$
\begin{equation}\label{penultima}
\theta(\sigma+it)\ll \exp((\log t)^{1-\sigma}\log\log\log t )\;a.s.
\end{equation}
By hypothesis we have that $\sum_{p\in\mathcal{P}}\frac{1+f(p)}{p^{1-\alpha+\epsilon}}$ converges $\forall\epsilon>0$ \textit{a.s.} Since $|1+f(p)|\leq 2$ $\forall p\in\mathcal{P}$, the Kolmogorov two series Theorem gives that $\sum_{p\in\mathcal{P}}\frac{1+\EE f(p)}{p^{1-\alpha+\epsilon}}$ converges $\forall\epsilon>0$. The construction made in the proof of Theorem \ref{alphaa1} gives that $\EE g(p)=1+\EE f(p)$ and hence that $\sum_{p\in\mathcal{P}}\frac{\EE g(p)}{p^{z}}$ is analytic in $\HH_{1-\alpha}$, since converges absolutely $\forall z\in\HH_{1-\alpha}$. This gives that $\EE G(z)$ is a Dirichlet series that converges absolutely in $\HH_{1-\alpha}$, and hence uniformly bounded on closed half planes. Moreover, by claim \ref{cl1} $\EE G(z)$ is non vanishing analytic function. We conclude that $G:\HH_{1-\alpha}\times\Omega\to\CC$ is a non-vanishing random analytic function, since $G(z)=\EE G(z) \theta(z)$, and satisfies for fixed $\sigma\in(1-\alpha,1]$:
\begin{equation}\label{ultima}
G(\sigma+it)\ll \exp((\log t)^{1-\sigma}\log\log\log t ),\;a.s.
\end{equation}

Let $T:\Omega\to\Omega$ be the measure preserving transformation as in the proof of Theorem \ref{alphaa1} and $\omega^*= T(\omega)$. The construction made in this proof gives $\forall \omega\in\Omega$ and $z\in\HH_1$ that
\begin{equation}\label{ultimato3}
\zeta(z)F_\omega(z)=U_{\omega^*}^{-1}(z)G_{\omega^*}(z).
\end{equation}
Assume $M_\mu(x)=o(x^{1-\alpha+\epsilon})$ for all $\epsilon>0$. In particular $\frac{1}{\zeta}$ is analytic in $\HH_{1-\alpha}$ and for each fixed $\sigma>1-\alpha$, $1/\zeta(\sigma+it)=o(t^\delta)$ for all $\delta>0$ (see \cite{tit} p. 336-337). Also, this implies that for \textit{almost all } $\omega\in\Omega$, $F_\omega$ has analytic extension to $\HH_{1-\alpha}$ given by
$$F_\omega(z)=\frac{G_{\omega^*}(z)}{U_{\omega^*}(z)\zeta(z)}.$$
By $(\ref{antepenultima})$, $(\ref{ultima})$ and the fact that $1/\zeta(\sigma+it)=o(t^\delta)$ for all $\delta>0$, for each fixed $\sigma>1-\alpha$, $F(\sigma+it)\ll t^{\delta}$ $\forall \delta>0$ \textit{a.s.}, and hence this implies (see the proof of Theorem \ref{alphaa}) that $M_f(x)=o(x^{1-\alpha+\epsilon})$ for all $\epsilon>0$, \textit{a.s.} On the other hand, if $M_f(x)=o(x^{1-\alpha+\epsilon})$ $\forall\epsilon>0$ \textit{a.s.}, by Lemma \ref{abell} $F$ has analytic extension to $\HH_{1-\alpha}$. Hence there is $\omega\in\Omega$ such that $U_{\omega^*}^{-1}(z)G_{\omega^*}(z)$ is a non vanishing analytic function in $\HH_{1-\alpha}$, and $F_\omega$ is analytic in this half plane. By (\ref{ultimato3}), we obtain that
$\zeta$ and $F_\omega(z)$ can not vanish in $\HH_{1-\alpha}\setminus\{1\}$, since they are analytic in this set and their product is a non-vanishing analytic function in $\HH_{1-\alpha}$. Hence $\frac{1}{\zeta}$ is analytic in $\HH_{1-\alpha}$ and satisfies in this half plane $\frac{1}{\zeta(\sigma+it)}=o(t^\delta)$ $\forall \delta>0$, which implies that $M_\mu(x)=o(x^{1-\alpha+\epsilon})$ for all $\epsilon>0$. \end{proof}

\section{Concluding Remarks}\label{concluding}

\noindent \textit{Asymptotic behavior of $M_{f_\alpha(x)}$}.

Let $0<\alpha<1/2$ and $f_\alpha$ be as in Theorem \ref{alphaa}. For $z\in\CC$ with $Re(z)>1-\alpha$, denote
$F_\alpha(z)=\sum_{n=1}^\infty \frac{f_\alpha(n)}{n^z}$. Then $\EE F_\alpha(z)= \frac{1}{\zeta(z+\alpha)}$.
A direct application of Lemma \ref{lema0} gives an analytic function $\theta_\alpha:\HH_{1/2}\times\Omega\to\CC$ such that
\begin{equation}\label{perron}
F_\alpha(z)=\frac{\theta_\alpha(z)}{\zeta(z+\alpha)},
\end{equation}
and therefore $F_\alpha$ extends analytically  to $\HH_{1/2}\setminus\{z\in\CC:\zeta(z+\alpha)= 0\}$  \textit{a.s.}
The closed half plane $\overline{\HH}_1$ is a zero free region for $\zeta$ and this implies the convergence of
$\sum_{n=1}^\infty \frac{\mu(n)}{n}$ (see \cite{pnt_newman} for a simple analytic proof and the references therein)
and hence that $M_\mu(x)=o(x)$. Moreover, this zero free region guarantees the existence of an open set containing $\HH_1$
where $\frac{1}{\zeta}$ is well defined and analytic. This fact combined with $(\ref{perron})$ gives the existence of
an open set $R_\alpha\supset \HH_{1-\alpha}$ such that $F_\alpha:R_\alpha\times\Omega\to\CC$ is a random analytic function.
This suggests that $\sum_{n=1}^\infty\frac{f_\alpha(n)}{n^{1-\alpha}}$ converges \textit{a.s.} We remark that the convergence
of this random series for $0<\alpha<1/2$ is, by Theorem \ref{alphaa}, a necessary condition for RH.

Further, we observe that the convergence of the series $F_\alpha(1-\alpha)$ is an tail event (see Proposition \ref{caldal}), in particular, by the Kolmogorov $0-1$ Law it has either probability $0$, or $1$. If we assume that $F_\alpha(1-\alpha)$ converges \textit{a.s.}, then the probability distribution of $F_{\alpha}(1-\alpha)$ is a dirac measure at $\{0\}$. Hence for each $p>x$, $f_{\alpha}(p)$ has no influence in the tail series $\sum_{n>x}\frac{f_\alpha(n)}{n^{1-\alpha}}$, while the value
$f_\alpha(p)$ for any $p\leq x$ is pivotal to evaluate this sum.

Our next result establishes this convergence in the range $0<\alpha<1/3$:
\begin{proposition}\label{vino} Let $0<\alpha<1/3$. Then for each $t\in \RR$ the random series
$$\sum_{n=1}^\infty \frac{f_\alpha(n)}{n^{1-\alpha+it}}$$
converges \textit{a.s.}
\end{proposition}
In particular, if $0<\alpha<1/3$ then $M_{f_\alpha}(x)=o(x^{1-\alpha})$ \textit{a.s.}
\begin{proof}[Proof of Proposition \ref{vino}] The Vinogradov-Korobov zero free region for $\zeta$ (see \cite{tit}, p. 135) is the set $R\subset \CC$ of the points $u+iv$ such that
$$1-u\leq \frac{A}{(\log v)^{2/3}(\log\log v)^{1/3}},$$ where $A>0$ is a constant. Moreover, in this region there exists a constant $B>0$ such that for all $u+iv\in R$ we have that
\begin{equation}\label{vino2}
\frac{1}{|\zeta(z)|}\leq C (\log v)^{2/3}(\log\log v)^{1/3}.
\end{equation}
Let $T>0$ and $x>0$ be arbitrarily large. For fixed $t\in\RR$ let
$$R_\alpha=R_\alpha(x,T,t):=\{u+iv\in\CC:u\in[1-\alpha-\delta,1+1/\log x],\; v\in[i(t-T),i(t+T)]\},$$ where
$$\delta=\delta(T):=\frac{A}{(\log (T+t))^{2/3}(\log\log (T+t))^{1/3}}.$$
Decompose $\partial R_\alpha = I_1\cup I_2\cup I_3\cup I_4$, where $I_1$ and $I_3$ are the vertical segments at
$Re(s)=1+1/\log x$ and $Re(s)=1-\alpha-\delta$ respectively, and $I_2$ and $I_4$ are the horizontal segments at
$Im(s)=T+t$ and $Im(s)=-T+t$ respectively. In the sequel $z=1-\alpha+it$ is a fixed complex number. For $1\leq k \leq 4$ let
$$J_k=\frac{1}{2\pi i}\int_{I_k}\frac{F(s)}{s-z}x^{s-z}ds.$$
By the Cauchy integral formula we have that $F(z)=\sum_{n=1}^4 J_n$. The Perron's formula (see \cite{tenenbaumlivro}, p. 133 Corollary 2.1)
gives that
$$J_1=\sum_{n\leq x}\frac{f_\alpha(n)}{n^z}-U(x,T),$$
where
$$U(x,T)\ll \frac{x^\alpha}{T}(\log (xT))+\frac{1}{x^{1-\alpha}},$$
and hence that
$$F(z)-\sum_{n\leq x}\frac{f_\alpha(n)}{n^z}=J_2+J_3+J_4-U(x,T).$$
To complete the proof we will show that $U(x,T)$ and $J_n$ ($2\leq n\leq 4$) become arbitrarily small as $x$ and $T$ become large in a specific way. Let
$$V(x,T):=\sup_{z\in\partial R_\alpha\setminus I_1} |F_\alpha(z)|.$$
By $(\ref{perron})$
$$V(x,T)\leq \sup_{z\in\partial R_\alpha\setminus I_1} |\theta_\alpha(z)|\sup_{z\in\partial R_\alpha\setminus I_1} |\zeta^{-1}(z+\alpha)|.$$
By $(\ref{vino2})$ we have that $\sup_{z\in\partial R_\alpha\setminus I_1} |\zeta^{-1}(z+\alpha)|\leq C (\log (T+t))^{2/3}(\log\log T+t)^{1/3}$.
On the other hand, an application of Theorem \ref{marcin} combined with Lemma \ref{lema0} gives that
\begin{align*}
\sup_{z\in\partial R_\alpha\setminus I_1} |\theta_\alpha(z)|\ll &\exp(\;(\log (T+t))^{\alpha-\delta} (\log\log (T+t))^{1-\alpha-\delta}\;)\\
                                                \ll& \exp(2\;(\log (T))^{\alpha} (\log\log (T))^{1-\alpha}\;), a.s.
\end{align*}
This estimates give that
\begin{equation}\label{vino3}
V(x,T)\ll (\log T)^{2/3}(\log\log T)^{1/3}\exp(2\;(\log T)^{\alpha} (\log\log T)^{1-\alpha}\;).
\end{equation}

\noindent \textit{Estimate for $J_2$ and $J_4$.} We have that
\begin{align*}
J_2 &=\int_{1-\alpha-\delta+i(t+T)}^{1+1/\log x+i(t+T)} \frac{F_\alpha(s)}{s-z}x^{s-z}ds\\
    &\ll V(x,T)x^{\alpha}\int_{1-\alpha-\delta}^{1+1/\log x} \frac{du}{|u-1+\alpha+iT|}\\
    & \ll V(x,T)\frac{x^\alpha}{T}.
\end{align*}
Similarly we get the same bound for $J_4$.

\noindent \textit{Estimate for $J_3$.} We have that
\begin{align*}
J_3 &   =   \int_{1-\alpha-\delta+it-iT}^{1-\alpha-\delta+it+iT}\frac{F(s)}{s-z}x^{s-z}ds \\
    & \ll   \frac{V(x,T)}{x^\delta}\int_{it-iT}^{it+iT}\frac{dy}{\sqrt{\delta^2+(y-t)^2}} \\
    & \ll   \frac{V(x,T)}{x^\delta}\log(T/\delta).\\
    & \ll \log(T/\delta)\exp\bigg{(}2\;(\log T)^{\alpha} (\log\log T)^{1-\alpha}-A\frac{\log x}{(\log T)^{2/3}(\log\log T)^{1/3} }  \bigg{)}
\end{align*}
In order to make $\lim_{x,T\to\infty} J_3=0$, the estimate above gives that $x$ and $T$ must be related in such a way that  $\frac{\log x}{(\log T)^{2/3}}$
grows faster than $(\log T)^{\alpha}$. This can be achieved by making  $\log x = (\log T)^\beta$ with $\beta>\alpha+2/3$. Using the estimate for $J_2$
we see that in order to make $\lim_{x,T\to\infty} J_2=0$ we need that
$$x^\alpha=o\bigg{(}\frac{T}{V(x,T)}\bigg{)}$$
and hence that $\beta\leq 1$. We conclude that in the range $\alpha<1/3$, the choice $x^{\alpha}=T^{1-\epsilon}$ for some fixed small
$\epsilon>0$ implies that $\lim_{x,T\to\infty}U(x,T)=0$ and that $\lim_{x,T\to\infty}J_n=0$ for $2\leq n\leq 4$, completing the proof.  \end{proof}

\noindent \textit{Remark of Theorem \ref{pretentious}.}

The proof of the first statement of Theorem \ref{pretentious} is an adaptation of the proof of Theorem \ref{viz1/2}. Next, we will state and prove a result that extends the second statement of this Theorem.

Let $f_\alpha$ be as in Theorem \ref{alphaa} and denote $f_{\alpha,1}=f_\alpha$. For $2\leq k\in\NN$ define $f_{\alpha,k}$
the binary random multiplicative function such that for $p\in\mathcal{P}$
$$\EE f_{\alpha,k}(p)=\begin{cases}-1,\mbox{ if }\frac{k}{p^\alpha}\geq 1;\\
                                   -\frac{k}{p^\alpha}, \mbox{ otherwise}.\end{cases}   $$
Then the Dirichlet series $F_{\alpha,k}(z)=\sum_{n=1}^\infty \frac{f_{\alpha,k}(n)}{n^z}$ converges for $Re(z)>1-\alpha$ and has a zero of multiplicity $k$ at $z=1-\alpha$ \textit{a.s.}

Observe that for weakly biased $f$ that satisfies $M_f(x)=o(x^{1-\alpha-\epsilon})$ for some $\epsilon>0$ \textit{a.s.}, and the series $\sum_{p\in \mathcal{P}}\frac{\EE f(p)}{p^{1-\alpha+\epsilon}}$ converges for $\epsilon>0$ and diverges to $-\infty$ for $\epsilon=0$, then $F(z)=\sum_{n=1}^\infty \frac{f(n)}{n^z}$ has a zero at $z=1-\alpha$ \textit{a.s.}, and hence this point is a zero of finite multiplicity of $F$ \textit{a.s.} The fact that $\forall z\in\HH_1$, $F(z)$ can be represented as the infinity product of independent random variables
$$F(z)=\prod_{p\in\mathcal{P}}\bigg{(}1+\frac{f(p)}{p^z}\bigg{)},$$
and for each $p\in\mathcal{P}$ and $z\in\HH_0$ we have that $\big{|}1+\frac{f(p)}{p^z}\big{|}>0$, by Proposition \ref{mensu} applied to $\frac{1}{F}$ we obtain that the multiplicity of any zero of $F$ in the half plane $\HH_{1/2}$ is an integer valued tail random variable, and hence it is constant \textit{a.s.} This suggests that $F$ behaves like $F_{\alpha,k}$ for some $k\geq 1$ and raises the question if $f$ could be related to $f_{\alpha,k}$. Our next result states:
\begin{theorem}\label{pretentiousk} Let $f$ be a weakly biased random multiplicative function such that $\EE f(p)=-\frac{\delta_p}{p^\alpha}$, where $0<\alpha<1/2$ and for some $k\in\NN$, $\forall{p\in\mathcal{P}}$ we have $\delta_p \leq k $. Assume that $z=1-\alpha$ is a zero of multiplicity $k$ of $F$ \textit{a.s.} Let $0<\epsilon<1/2-\alpha$. Then the assumption $M_f(x)=o(x^{1-\alpha-\epsilon})$ \textit{a.s.} implies that $\mathbb{D}_{\sigma}(f,f_{\alpha,k})<\infty$ for all $\sigma>1-\alpha-\epsilon$ \textit{a.s.} and that $\zeta$ has no zeros in $\HH_{1-\epsilon}$.
\end{theorem}
\begin{proof} We assume that $\forall p\in\mathcal{P}$, $f(p)$ and $f_{\alpha,k}(p)$ are given by $(\ref{indicator})$. In particular, $f$ and $f_{\alpha,k}$ are uniformly coupled. Let $g$ be a random multiplicative function such that $\forall p\in\mathcal{P}$, $g(p)$ is given by $(\ref{indicator})$ and
\begin{equation}\label{pret}
\EE g(p)=2(\PP(f_{\alpha,k}(p)=-1)-\PP(f(p)=-1))=1-\EE f(p)f_{\alpha,k}(p).
\end{equation}
Let $u$ be an unbiased binary random multiplicative function given by $(\ref{indicator})$. Let $U$, $F$, $F_{\alpha,k}$ and $G$ be the associated Dirichlet series of $u$, $f$, $f_{\alpha,k}$ and $g$ respectively. For $p\in\mathcal{P}$ let $I_p$ and $J_p$ be the intervals $(\PP(f(p)=-1),\PP(f_{\alpha,k}(p)=-1)]$ and $(\PP(g(p)=-1),1/2]$, respectively. Observe that $I_p$ and $J_p$ have the same Lebesgue measure. Let $\varphi=\frac{F}{F_{\alpha,k}}$ and $\psi=\frac{G}{U}$. Then there exists a measure preserving transformation $T$ such that for each $\omega\in\Omega$ we have $\varphi_\omega=\psi_{T\omega}$ (see the proof of Theorem \ref{alphaa1}). Arguing in the same way as in the proof of Theorem \ref{alphaa1} we conclude that $G$ extends analytically to $\HH_{1-\alpha-\epsilon}$ \textit{a.s.} and this implies, by Lemma \ref{left}, that the series $\sum_{p\in\mathcal{P}}\frac{g(p)}{p^{\sigma}}$ converges for all $\sigma>1-\alpha-\epsilon$ \textit{a.s.} This fact and the Kolmogorov Two-Series Theorem combined with $(\ref{pret})$ gives that $\mathbb{D}_\sigma(f,f_{\alpha,k})<\infty$ for all $\sigma>1-\alpha-\epsilon$ \textit{a.s.} Moreover the \textit{a.s.} convergence of $\sum_{p\in\mathcal{P}}\frac{g(p)}{p^{z}}$ in $\HH_{1-\alpha-\epsilon}$ implies that $G$ is a non-vanishing random analytic function in $\HH_{1-\alpha-\epsilon}$ so as $\psi$. We conclude that $\varphi$ is a non-vanishing random analytic function in $\HH_{1-\alpha-\epsilon}$. By Lemma \ref{lema0} there exists a non-vanishing random analytic function $\Lambda:\HH_{1/2}\times\Omega\to\CC$ such that $F_{\alpha,k}(z)\zeta^{k}(z+\alpha)=\Lambda(z)$, $\forall z\in\HH_1$. Therefore for $Re(z)>1$ we have:
$$F(z)\zeta^k(z+\alpha)=\varphi(z)\Lambda(z).$$
Since $\varphi\Lambda\neq 0$ in $\HH_{1/2}$ \textit{a.s.} and $F(z)\zeta^{k}(z+\alpha)$ is the product of two random analytic functions in $\HH_{1-\alpha-\epsilon}\setminus\{1-\alpha\}$ we have that none of them can vanish in this set. In particular $\zeta$ is free of zeros in $\HH_{1-\epsilon}$.
\end{proof}

\textbf{Acknowledgements.} Authors thank K. Soundararajan and S. Chatterjee for fruitful discussions. We also thank K.Soundararajan for pointing
out and explaining work of D. Koukoulopoulos. M.A. thanks C. G. Moreira  for fruitful discussions.

\noindent We thank MSRI (Berkeley) and Stanford University for hospitality and financial support during multiple visits. Both authors were supported by IMPA, Cnpq and Faperj.

\appendix
\section{}

\begin{theorem}\label{abell} Let $f:\NN\to[0,1]$ be such that the Dirichlet series $\sum_{k=1}^\infty\frac{f(k)}{k^z}$ converges absolutely for $z\in\HH_a$ ($a>0$) and that for some $0<c<a$
$$\sup_{x\geq 1}\frac{|M_f(x)|}{x^c}=C<\infty.$$
Then the Dirichlet series $\sum_{k=1}^\infty\frac{f(k)}{k^z}$ converges for all $z\in\HH_c$ and $F:\HH_c\to\CC$ given by $F(z):=\sum_{k=1}^\infty\frac{f(k)}{k^z}$ is analytic.
\end{theorem}
\begin{proof} Let $y<x$ be natural numbers. Let $\{a_k\}_{k\in\NN}$ and  $\{b_k\}_{k\in\NN}$ be two sequences of real numbers and let $\Delta a_k := a_k-a_{k-1}$. The partial summation formula states that:
\begin{equation}\label{partes}
\sum_{k=y}^x a_k\Delta b_k = a_xb_x-a_yb_{y-1}-\sum_{k=y}^{x-1}b_k\Delta a_{k+1}.
\end{equation}
Let $\delta>0$. A direct application of $(\ref{partes})$ gives that:
\begin{align*}
\bigg{|}\sum_{k=y}^x\frac{f(k)}{k^{c+\delta}}\bigg{|}&=\bigg{|}\sum_{k=y}^x\frac{\Delta M_f(k)}{k^{c+\delta}}\bigg{|}\\
&= \bigg{|}\frac{M_f(x)}{x^{c+\delta}}-\frac{M_f(y)}{(y-1)^{c+\delta}}-\sum_{k=y}^
{x-1}M_f(k)\Delta\frac{1}{(k+1)^{c+\delta}}\bigg{|}
\end{align*}
By hypothesis $\lim_{x\to\infty}\frac{M_f(x)}{x^{c+\delta}}=0$ and $\lim_{y\to\infty}\frac{M_f(y)}{x^{c+\delta}}=0$.
Since $\Delta\frac{1}{(k+1)^{c+\delta}}=-(c+\delta)\int_{k}^{k+1}\frac{dt}{t^{1+c+\delta}}\ll \frac{1}{k^{1+c+\delta}}$, we obtain
$$\bigg{|}\sum_{k=y}^{x-1}M_f(k)\Delta\frac{1}{(k+1)^{c+\delta}}\bigg{|}\ll \sum_{k=y}^x \frac{1}{k^{1+\delta}}=o(y).$$
We conclude that Dirichlet series $F(c+\delta)$ is convergent for every $\delta>0$. A classical result in the Theory of the Dirichlet series (see \cite{apostol}, Theorems 11.8 and 11.11) states that if the series $\sum_{k=1}^\infty \frac{f(k)}{k^{z}}$ converges for $z_0=\sigma_0+it_0$ then it converges for all $z\in\HH_{\sigma_0}$ and also uniformly on compact subsets of this half plane. Thus the function $z\in\HH_{\sigma_0}\mapsto\sum_{k=1}^\infty \frac{c_k}{k^z}$ is analytic. \end{proof}
\begin{proposition}\label{caldal0} Let $f$ be a random multiplicative function. Then for each $c>1/2$
$$E_{f,c}:=\bigg{\{}\omega\in\Omega: \sum_{k=1}^\infty \frac{f_\omega(k)}{k^c}\mbox{ converges }\bigg{\}}$$
is an tail event and hence, $\PP(E_{f,c})\in\{0,1\}$.
\end{proposition}
\begin{proof} Let $k\in\NN$ and $D=D(k):=\{p\in\mathcal{P}:p\leq k\}$. Let $f:\NN\to\{-1,0,1\}$ and $h_D:\NN\to\{-1,0,1\}$ be multiplicative functions (supported on the squarefree integers) such that
$$h_D(p)=\begin{cases} 0,&\mbox{ if }p\in D,\\
f(p),&\mbox{ if }p\notin D,   \end{cases}$$
\begin{claim}\label{caldal1} The series $\sum_{n=1}^\infty \frac{f(n)}{n^c}$ converges if and only if $\sum_{n=1}^\infty \frac{h_D(n)}{n^c}$ converges.
\end{claim}

\noindent\textit{Proof of the claim}: Let $u$ be a multiplicative function supported on the set of the squarefree integers such that
$$u(p)=\begin{cases} f(p),&\mbox{ if }p\in D,\\
0,&\mbox{ if }p\notin D.   \end{cases}$$
Then for all $p\in\mathcal{P}$ $(u\ast h_D)(p)=u(p)+h_D(p)=f(p)$ and for every $l\geq 2$
$$(u\ast h_D)(p^l)=\sum_{k=0}^lu(p^k)h_D(p^{l-k})=u(1)h_D(p^l)+u(p)h_D(p^{l-1})=0,$$
since $u(p^m)=h_D(p^m)=0$ for every $m\geq 2$ and $u(p)h_D(p)=0$. This shows that $u\ast h_D (p^m)=f(p^m)$ $\forall p\in\mathcal{P}$ and $\forall m\geq1$. Since $u$ and $h_D$ are multiplicative, their convolution also is and hence
$f(n)=h_D\ast u(n)$ $\forall n\in\NN$. Let $U(z)$ be the Dirichlet series of $u$. Then $U$ has Euler product representation
$$U(z)=\prod_{p\in D}\bigg{(}1+\frac{u(p)}{p^z}\bigg{)}\quad (z\in\HH_1),$$
and since $D$ is finite, we obtain that the Dirichlet series $U(z)$ converges absolutely for all $z\in\HH_0$ (\cite{tenenbaumlivro} page 106, Theorem 2). Hence the convergence of $\sum_{n=1}^\infty \frac{h_D(n)}{n^c}$ implies the convergence $\sum_{n=1}^\infty \frac{f(n)}{n^c}$ (\cite{tenenbaumlivro} p. 122, Notes 1.1). On the other hand, let $u^{-1}$ be the Dirichlet inverse of $u$, that is, $(u\ast u^{-1})(n)=\ind_{\{1\}}(n)$. Then $u^{-1}$ is multiplicative, $|u^{-1}(n)|\leq 1$ $\forall n\in\NN$ and $\sum_{n=1}^\infty\frac{u^{-1}(n)}{n^z}=\frac{1}{U(z)}$ ($z\in\HH_1$). Moreover:
$$\frac{1}{U(z)}=\prod_{p\in D}\bigg{(}1+\sum_{m=1}^\infty(-1)^m \frac{u(p)^m}{p^{mz}}\bigg{)}.$$
Since $D$ is finite and
$$\sum_{p\in D}\sum_{m=1}^\infty (-1)^m \frac{u(p)^m}{p^{mz}}=\sum_{p\in D}\frac{1}{1+\frac{u(p)}{p^{z}}},$$
we obtain that $\sum_{p\in D}\sum_{m=1}^\infty (-1)^m \frac{u(p)^m}{p^{mz}}$ converges absolutely in $\HH_0$.
This implies that, $\sum_{n=1}^\infty\frac{u^{-1}(n)}{n^z}$ converges absolutely in $\HH_0$ (\cite{tenenbaumlivro} page 106, Theorem 2). Since $h_D=f\ast u^{-1}$ we obtain that the convergence of $\sum_{n=1}^\infty \frac{f(n)}{n^c}$ implies the convergence of $\sum_{n=1}^\infty \frac{h_D(n)}{n^c}$ (\cite{tenenbaumlivro} p. 122, Notes 1.1), completing the proof of the claim.

\noindent Let $\mathcal{F}_n^\infty$ be the sigma algebra generated by the random variables $\{f(p):p\in\mathcal{P}\mbox{ and }p\geq n\}$. The tail sigma algebra of $\mathcal{F}$, denoted by $\mathcal{F}^*$ is the sigma algebra
$$\mathcal{F}^*=\bigcap_{n=1}^\infty \mathcal{F}_n^\infty.$$
Elements of $\mathcal{F}^*$ are called tail events. The Kolmogorov zero or one law states that every tail event has either probability zero or one. Recall that $D=D(k)$ and $E_{h_D,c}\in\mathcal{F}_k^\infty$. The claim \ref{caldal1} gives that $E_{f,c}=E_{h_D,c}$. In particular, $E_{f,c}\in \mathcal{F}_k^\infty$ $\forall k\in\NN$. \end{proof}

\begin{corollary}\label{caldal} Let $0<\alpha<1/2$. The following are tail events:
$$[M_f(x)=o(x^{1-\alpha+\epsilon})\mbox{ for all }\varepsilon>0 ].$$
$$[M_f(x)=o(x^{1-\alpha-\epsilon})\mbox{ for some }\varepsilon>0 ].$$
\end{corollary}
\begin{proof} Let $E_{f,c}$ be as in Proposition \ref{caldal1}. Recall that if the Dirichlet series $\sum_{k=1}^\infty\frac{f(k)}{k^{\sigma_0}}$ converges, then it converges for all $\sigma>\sigma_0$. Thus, by Kroencker's Lemma and by partial summation (see the proof Theorem \ref{abell}):
\begin{align*}
&[M_f(x)=o(x^{1-\alpha+\epsilon})\mbox{ for all }\varepsilon>0 ]=\bigcap_{n=1}^\infty E_{f,\alpha+n^{-1}},\\
&[M_f(x)=o(x^{1-\alpha-\epsilon})\mbox{ for some }\varepsilon>0 ]=\bigcup_{n=1}^\infty E_{f,\alpha-n^{-1}}.
\end{align*}
\end{proof}
\begin{lemma}\label{polo} Let $f:\NN\to[0,1]$ and for $x>0$, $L(1+x)=\sum_{k=1}^{\infty}\frac{f(k)}{k^{1+x}}$. Let $a,b:[0,\infty)\to (0,1]$ be such that
$a(t)\leq b(t)$ for all $t$, $\lim _{t\to\infty}b(t)=0$ and $b(t)-a(t)\ll a^2(t)$. Then as $t\to\infty$:
$$L(1+a(t))=L(1+b(t))+O(1).$$
\end{lemma}
\begin{proof} Denote $a=a(t)$ and $b=b(t)$. Let $k\in\NN$ and $\psi_k(x)=\exp(-x\log k)$. Hence
$$|\psi_k(a)-\psi_k(b)|\leq \int_{a}^b |\psi_k'(x)|dx=\log k\int_{a}^b |\psi_k(x)|dx\leq (b-a)\psi_k(a).$$
Let $x>0$ and $\zeta(1+x)=\sum_{k=1}^\infty\frac{1}{k^{1+x}}$. If $x>0$ is small, a well known fact is that the Riemann $\zeta$ is a meromorphic function with a simple pole at $z=1$ with residue $1$. Hence for $x>0$, $\sum_{k=1}^\infty \frac{\log k}{k^{1+x}}=|\zeta'(1+x)|\sim \frac{1}{x^2}$. This combined with the estimative for $\psi_k$ gives:
\begin{align*}
|L(1+a)-L(1+b)|=&\bigg{|}\sum_{k=1}^\infty \bigg{(}\frac{f(k)}{k^{1+a}}-\frac{f(k)}{k^{1+b}}\bigg{)}\bigg{|}\\
\leq& \sum_{k=1}^\infty \frac{f(k)}{k}|\psi_k(a)-\psi_k(b)|\\
\leq& (b-a)\sum_{k=1}^\infty \frac{\log k}{k^{1+a}}\\
=&|(b-a)\zeta'(1+a)|\\
\ll& \frac{(b-a)}{a^2}=O(1).
\end{align*}
\end{proof}
\subsection{Extension of random analytic functions to half planes.}\label{era}
\begin{proposition}\label{mensu} Let $(\Omega,\mathcal{F},\PP)$ be a probability space and $f:\HH_1\times\Omega\to\CC$ be a random function such that
$f_\omega:\HH_1\to\CC$ is analytic for all $\omega\in\Omega$. Then, for each fixed $c<1$, the following set is an element of $\mathcal{F}$:
$$A_f:=\{\omega\in\Omega:f_\omega\mbox{ has analytic extension to  }\HH_{c}\}.$$
\end{proposition}
\begin{proof} Since $f$ is a random function, for fixed $z\in\HH_1$, the complex random variables $f(z)$ and $\{f(z+k^{-1})\}_{k\in\NN}$
are $\mathcal{F}-$measurable. For each $\omega\in\Omega$ we have that $f_{\omega}:\HH_1\to\CC$ is analytic. Hence, for all $z\in\HH_1$ the limit
\begin{equation}\label{deriv}
f_{\omega}^{(1)}(z):=\lim_{k\to\infty} \frac{f_{\omega}(z+k^{-1})-f_{\omega}(z)}{k^{-1}}
\end{equation}
exists and it is a complex random variable measurable in $\mathcal{F}$. Hence $f^{(1)}:\HH_1\times\Omega\to\CC$ is a random function which is analytic for all $\omega\in\Omega$, since for each $\omega\in\Omega$, it is the derivative of $f_{\omega}$. By applying these arguments inductively, we conclude that for all $n\in\NN$, $f^{(n)}:\HH_1\times\Omega\to\CC$ given by $f_{\omega}^{(n)}:=\frac{d^n}{dz^n}f_{\omega}(z)$ is a random analytic function such that for each $z\in\HH_1$, $f^{(n)}(z)$ is $\mathcal{F}-$measurable. Denote $B(a,\delta):=\{z\in\CC:|z-a|<\delta\}$. We recall the following result from Complex-Analysis (c.f \cite{conway}, page 72, Theorem 2.8):
\begin{claim}\label{power}Let $G$ be an open connected set, $h:G\to\CC$ an analytic function, $a\in G$ and $R>0$ such that $B(a,R)\subset G$. Then for all $z\in B(a,R)$ we have that $h(z)=\sum_{n=0}^\infty \frac{h^{(n)}(a)}{n!}(z-a)^n$. Moreover the radius of convergence of this power series is greater or equal to $R$.
\end{claim}
Denote $f^{(0)}=f$ and for each $k\in\NN$,
$$\frac{1}{R_k(\omega)}:=\limsup_{n\to\infty}\bigg{|}\frac{1}{n!}f_{\omega}^{(n)}(k+1)\bigg{|}^{\frac{1}{n}}.$$
Hence $R_k$ also is $\mathcal{F}-$measurable. We recall that $R_k(\omega)$  is the radius of convergence of the power series of the complex analytic function $f_{\omega}$ at the point $z=k+1$ (c.f. \cite{conway}, chapter III ).
We claim that:
\begin{equation}\label{mens}
A_f=\bigcap_{k\in\NN}[R_k \geq k+1-c].
\end{equation}
A direct application of the claim \ref{power} gives that $A_f\subset\bigcap_{k\in\NN}[R_k \geq k+1-c]$, since $B(k+1,k+1-c)\subset \HH_c$, for all $k\in\NN$. To prove the other inclusion, let $A_k:=B(k+1,k+1-c)\setminus{\HH_1}$. Observe that $A_k\subset A_{k+1}$ for all $k\in\NN$ and that $\HH_c=\bigcup_{k\in\NN}\HH_1\cup A_k$. For each $\omega\in \bigcap_{k\in\NN}[R_k \geq k-c]$ and $k\in\NN$ define $H_{\omega,k}:B(k,k+1-c)\to\CC$ by
$$H_{\omega,k}(z):=\sum_{n=0}^\infty\frac{f_{\omega}^{(n)}(k+1)}{n!}(z-(k+1))^n.$$
By Claim \ref{power}, the assumption that $R_k(\omega)\geq k+1-c$ gives that $H_{\omega,k}$ is analytic in the open ball $B(k+1,k+1-c)$ and  $H_{\omega,k}(z)=f_\omega(z)$ for each $z\in\HH_1\cap B(k+1,k+1)$. Hence $H_{\omega,k}(z)=f_\omega(z)$ for all $z\in\HH_1\cap B(k+1,k+1-c)$. This follows from the fact that, if two analytic functions defined in a open connected set $R_1$ coincide in an open ball $B\subset R_1$, then these analytic functions coincide in all $R_1$ (see \cite{conway} Theorem 3.7, pg 78). Hence $G_{\omega,k}:\HH_1\cup A_k\to\CC$ given by
$$G_{\omega,k}(z):=\begin{cases} f_{\omega}(z),&\mbox{ if }z\in\HH_1\setminus B(k+1,k+1-c);\\
H_{\omega,k}(z),&\mbox{ if }z\in B(k+1,k+1-c)\end{cases}$$
is an analytic extension of $ f_{\omega}$ to $\HH_1\cup A_k$. Since this is an open connected set of $\CC$, $G_{\omega,k}$ is the unique analytic function defined in $\HH_1\cup A_k$ that coincides with $f_{\omega}$ in $\HH_1$. Observe that for each $k\in\NN$, $\HH_1\cup{A_k}\subset\HH_{1}\cup A_{k+1}$ and these are open connected sets. Hence the unique analytic extension of $G_{\omega,k}$ to $\HH_{1}\cup A_{k+1}$ is $G_{\omega,k+1}$. This implies that for each $\omega\in \bigcap_{k\in\NN}[R_k \geq k-c]$, $G_\omega:\HH_c\to\CC$ given by
$$G_\omega(z):=\begin{cases}f_\omega(z),&\mbox{ if }z\in\HH_1;\\
                             G_{\omega,k}(z),&\mbox{ if } z\in A_k,\mbox{ for }k\in\NN                         \end{cases}$$
is well defined and analytic. By well defined we mean that for each $z \in \HH_c \setminus \HH_1$, there exists $k_0$ such that $z\in A_k$ for all $k\geq k_0$ and the value $G_\omega(z)=G_{\omega,k}(z)$ does not depend on $k$. The analyticity  of $G_\omega$  follows from the fact that for each fixed $z\in\HH_c$, there is an small $\delta>0$ and $k\in\NN$ such that $B(z,\delta)\subset \HH_1 \cup A_k$. Hence for all $w \in B(z,\delta)$, $G_\omega(w)=G_{\omega,k}(w)$. Since $G_{\omega,k}$ is holomorphic at $z$, $G_\omega$ also is holomorphic at $z$. Hence $G_\omega$ is holomorphic for all $z\in\HH_c$. This completes the proof of $(\ref{mens})$ which gives that $A_f$ is the countable intersection of $\mathcal{F}-$measurable sets. \end{proof}

{\small{\sc Departamento de Matem\'atica, Universidade Federal de Minas Gerais, Av. Ant\^onio Carlos, 6627, CEP 31270-901, Belo Horizonte, MG, Brazil.} \\
\textit{Email address:} aymone.marco@gmail.com}

{\small{\sc Instituto Nacional de Matem\'atica pura e Aplicada, Estrada Dona Castorina 110, Jardim Bot\^anico, CEP 22460-320, Rio de Janeiro, RJ, Brazil.} \\
\textit{Email address:} vladas@impa.br}

\end{document}